\documentclass[a4paper,english]{amsart}
\usepackage[utf8]{inputenc}
\usepackage[english]{babel}
\usepackage{amsmath,amsthm,amssymb,tikz}
\usepackage{hyperref, enumerate}
\usepackage{graphicx}
\usepackage[all,2cell]{xy}
\usepackage{comment}
\usepackage{tikz-cd}
\usepackage{adjustbox}
\usepackage{undertilde}
\usepackage{xcolor,colortbl}

\theoremstyle{plain}
\newtheorem{theorem}{Theorem}[section]
\newtheorem{lemma}[theorem]{Lemma}
\newtheorem{proposition}[theorem]{Proposition}
\newtheorem{corollary}[theorem]{Corollary}
\newtheorem{example}[theorem]{Example}

\theoremstyle{definition}
\newtheorem{definition}[theorem]{Definition}

\newtheorem{remark}[theorem]{Remark}

\newcommand{\alg}[1]{\mathbf{#1}}	
\newcommand{\var}[1]{\mathcal{#1}}
\newcommand{\func}[1]{\mathsf{#1}}
\newcommand{\lucas}{\text{\bf\L}}
\newcommand{\PL}{\alg{P \lucas}}
\newcommand{\Stone}{\mathsf{Stone}}
\newcommand{{\MV}}{\mathsf{MV}}
\newcommand{{\DL}}{\mathsf{DL}}
\newcommand{{\BA}}{\mathsf{BA}}
\newcommand{{\LM}}{\mathsf{LM}}
\newcommand{{\PMV}}{\mathsf{PMV}}
\newcommand{{\Priest}}{\mathsf{Priest}}
\newcommand{\variety}[1]{\mathbb{H}\mathbb{S}\mathbb{P}(#1)}
\newcommand{\qvariety}[1]{\mathbb{I}\mathbb{S}\mathbb{P}(#1)}
\newcommand{\tqvariety}[1]{\mathbb{I}\mathbb{S}_c\mathbb{P}^+(#1)}
\newcommand{\Pow}{\mathfrak{P}}
\newcommand{\Skel}{\mathfrak{S}}
\newcommand{\pr}{\mathsf{pr}}

\newcommand{\id}{\mathsf{id}}
\newcommand{\U}{\mathsf{U}}

\begin{document}

\title[Natural dualities for varieties gen. by finite positive MV-chains]{Natural dualities for varieties generated by finite positive MV-chains}

\author[W. Poiger]{Wolfgang Poiger}
\address{University of Luxembourg \\ 6 Avenue de la Fonte \\ L-4364 Esch-sur-Alzette \\ Luxembourg}	
\email{wolfgang.poiger@uni.lu}


\begin{abstract}
We provide a simple natural duality for the varieties generated by the negation- and implication- free reduct of a finite MV-chain. We study these varieties through the dual equivalence thus obtained. For example, we fully characterize their algebraically closed, existentially closed and injective members. We also explore the relationship between this natural duality and Priestley duality in terms of distributive skeletons and Priestley powers.                  
\end{abstract}
\maketitle
\section{Introduction}\label{sec:Introduction}
Just like distributive lattices are the negation-free subreducts of Boolean algebras, \emph{positive $\MV$-algebras} (recently introduced in \cite{Abbadini2022}),  are the negation-free subreducts of $\MV$-algebras. While the variety $\MV$ of $\MV$-algebras provides algebraic semantics for \L ukasiewicz infinite-valued logic (see, \emph{e.g.}, \cite[Chapter 4]{Cignoli2000}), its subvarieties $\MV_n = \variety{\lucas_n}$, generated by finite $\MV$-chains $\lucas_n$, provide algebraic semantics for \L ukasiewicz finitely-valued logics (the subvarieties $\MV_n$ were first studied in \cite{Grigolia1977}). In this paper, we study the varieties $\PMV_n = \variety{\PL_n}$ generated by finite \emph{positive} $\MV$-chains $\PL_n$ (that is, negation-free reducts of $\lucas_n$). Our main tool for this study is the theory of \emph{natural dualities}. 

In its simplest form, natural duality theory \cite{ClarkDavey1998} provides a general framework to obtain a dual equivalence between a quasi-variety $\var{A} = \qvariety{\alg{M}}$ generated by a finite algebra $\alg{M}$ and a category of structured Stone spaces $\var{X} = \tqvariety{\utilde{\alg{M}}}$ generated by a discrete structure $\utilde{\alg{M}}$ based on the same set as $\alg{M}$ (and therefore called an \emph{alter ego} of $\alg{M}$). Examples of natural dualities are \emph{Stone duality}, which arises if $\alg{M}$ is the two-element Boolean algebra and its alter ego $\utilde{\alg{M}}$ is the two-element discrete space, and \emph{Priestley duality}, which arises if $\alg{M}$ is the two-element distributive lattice and its alter-ego $\utilde{\alg{M}}$ is the two-element discrete space with order $0 \leq 1$. The great utility of these dualities may be seen as a consequence of the fact that the alter-egos described above are very simple structures. A general theme of natural duality theory may be phrased as \emph{simple structures yield useful dualities}.       

For the varieties $\MV_n$, natural dualities were developed and studied in \cite{Niederkorn2001}. Since finite $\MV$-chains are \emph{semi-primal}, it is easy to come up with simple alter-egos $\utilde{\lucas}_n$. Indeed, by \cite[Theorem 3.3.14]{ClarkDavey1998} the only structure relevant for this duality is given by the collection of subalgebras $\mathbb{S}(\lucas_n)$. One reason why this is sufficient is that every subalgebra of $\lucas_n \times \lucas_n$ is a direct product of subalgebras. This is \emph{not} true anymore in the case of $\PL_n$ (for example, the order relation $\leq$ itself is a subalgebra of $\PL_n \times \PL_n$). We show that, instead, a simple alter-ego of $\PL_n$ can be obtained from a certain collection of subalgebras of the order $\leq$ which can be easily computed algorithmically. 

To make the case for the utility of these alter-egos, we investigate the dualities they yield to derive various results about the varieties $\PMV_n$. For example, we completely characterize the injective, algebraically closed and existentially closed members of $\PMV_n$. We also explore the relationship to Priestley duality, which can be expressed in terms of \emph{distributive skeletons} and \emph{Priestley powers}. We show that these constructions give rise to an adjunction between $\DL$ (the variety of distributive lattices) and $\PMV_n$, similar to the adjunction \cite[Section 4]{KurzPoigerTeheux2023} between $\BA$ (the variety of Boolean algebras) and $\MV_n$ obtained from functors taking the Boolean skeleton and the Boolean power, respectively. We expect this to prove useful in future applications, exploring modal logic over $\PMV_n$ as an analogue of Dunn's \emph{positive modal logic} \cite{Dunn1995} in the setting of modal finitely-valued \L ukasiewicz logic \cite{HansoulTeheux2013}.

The paper is structured as follows. In Section~\ref{sec:Preliminaries}, we recall the most important background information on $\MV$- and $\MV_n$-algebras (Subsection~\ref{subsec:MV-Algebras}) and on natural dualities (Subsection~\ref{subsec:NaturalDualities}). In Section~\ref{sec:NatDualposMv}, we begin our study of finite positive $\MV$-chains $\PL_n$ and the varieties $\PMV_n$ they generate (Subsection~\ref{subsec:PositiveMVChains}). We proceed to develop the natural dualities for these varieties (Subsection~\ref{subsec:NaturalDualitiesPosMV}). In Section~\ref{sec:FurtherExplorations}, we explore some ramifications of this duality. More specifically, we give an explicit axiomatization of the category dual to $\PMV_2$ generated by the three-element chain (Subsection~\ref{subsec:ThreeElementChain}), we explore the relationship between $\PMV_n$ and $\DL$ as described above (Subsection~\ref{subsec:RelationshipPriestley}) and we characterize the algebraically and existentially closed algebras in $\PMV_n$ via this relationship (Subsection~\ref{subsec:A-E-ClosedAlgebras}). In the concluding Section~\ref{sec:conclusion}, we collect some open questions and ideas for further research.  
\section{Preliminaries}\label{sec:Preliminaries}
In this section, we give short overviews of the two most important topics related to this paper. In Subsection~\ref{subsec:MV-Algebras}, we recall the basics of $\MV$-algebras, with a focus on finite $\MV$-chains and the varieties $\MV_n$ they generate. In Subsection~\ref{subsec:NaturalDualities}, we recall important prerequisites from the theory of natural dualities. For further information on these topics, the reader may consult the textbooks \cite{Cignoli2000} about $\MV$-algebras and \cite{ClarkDavey1998} about natural dualities (in particular, we will often refer to the latter throughout this entire paper).     
\subsection{MV-algebras}\label{subsec:MV-Algebras}
It is well-known that Boolean algebras provide an appropriate algebraic counterpart to classical propositional logic. Similarly, to \L ukasiewicz infinitely-valued logic, an appropriate algebraic counterpart is provided by \emph{$\MV$-algebras}, introduced by Chang \cite{Chang1958} in 1958. The variety $\MV$ of $\MV$-algebras is generated by the \emph{standard $\MV$-algebra}
\[
\langle [0,1], \odot, \oplus, \wedge, \vee, \neg, 0, 1 \rangle
\] 
based on the real unit interval with its usual bounded lattice structure and additional operations 
\[
 x\odot y = \mathsf{max}\{0, x + y -1\}, \text{ } x\oplus y = \mathsf{min}\{1, x + y \}, \text{ } \neg x = 1 - x.
\]        
For a detailed overview of $\MV$-algebras and their relationship to many-valued logic, we refer the reader to \cite{Cignoli2000} (and \cite{Mundici2011} for more advanced topics). In this paper, we focus on the finite subalgebras of the standard $\MV$-algebra, which are all of the following form.
\begin{definition}\label{def:FiniteMVChains}
Let $n\geq 1$ be a natural number. The \emph{$(n+1)$-element $\MV$-chain} is given by 
\[
\lucas_n = \langle \{ 0, \tfrac{1}{n},\dots \tfrac{n-1}{n}, 1 \}, \wedge , \vee, \odot, \oplus, \neg, 0, 1\rangle, 
\] 
considered as a subalgebra of the standard $\MV$-algebra. We use $\MV_n$ to denote the variety $\variety{\lucas_n}$ generated by $\lucas_n$ (these varieties were first axiomatized by Grigolia in \cite{Grigolia1977}).  
\end{definition}
Note that $\lucas_1$ is simply the two-element Boolean algebra and, therefore, $\MV_1$ is the variety of Boolean algebras, the algebraic counterpart to classical propositional logic. The varieties $\MV_n$ with $n \geq 2$ provide appropriate algebraic counterparts to \L ukasiewicz \emph{finitely-valued logics}. In  particular, \L ukasiewicz three-valued logic with $\lucas_2$ as algebra of truth-degrees is a popular research topic in non-classical logic.  
             
It was shown in \cite[Proposition 2.1]{Niederkorn2001} that every finite $\MV$-chain $\lucas_n$ is \emph{semi-primal} \cite{FosterPixley1964a}, meaning that every operation $f\colon \{ 0,\frac{1}{n},\dots, \frac{n-1}{n}, 1\}^k \to \{ 0,\frac{1}{n},\dots, \frac{n-1}{n}, 1\}$ ($k\geq 1$) which preserves subalgebras of $\lucas_n$ (that is, $f(\alg{S}^k) \subseteq \alg{S}$ for all subalgebras $\alg{S} \subseteq \lucas_n$) is term-definable in $\lucas_n$ (note that this is a straightforward generalization of the two-element Boolean algebra $\lucas_1$ being \emph{primal}, meaning that \emph{every} operation $f\colon \{ 0,1 \}^k \to \{0,1\}$ can be expressed by a Boolean term). Some important consequences of this are that the variety $\MV_n$ coincides with the quasi-variety $\qvariety{\lucas_n}$ generated by $\lucas_n$, and that every $\MV_n$-algebra is a \emph{Boolean product} (see, \emph{e.g.}, \cite[Chapter IV]{BurrisSankappanavar1981}) of subalgebras of $\lucas_n$. 
      
It is well-known that the subalgebras of $\lucas_n$ are exactly given by 
\[
\lucas_k \cong \langle \{ 0, \tfrac{\ell}{n}, \dots, \tfrac{(k-1)\ell}{n}, 1\}, \wedge , \vee, \odot, \oplus, \neg, 0, 1\rangle,
\]
where $n = k\cdot\ell$. Therefore, the lattice $\mathbb{S}(\lucas_n)$ of subalgebras of $\lucas_n$ is isomorphic to the bounded lattice of divisors of $n$. 

It is also well-known (and another immediate consequence of $\lucas_n$ being semi-primal) that, for every $d\in \lucas_n$, the unary operation 
\[
\tau_d(x) = 
\begin{cases}
1 & \text{ if } d \leq x, \\
0 & \text { otherwise } 
\end{cases}  
\]  
is term-definable in $\lucas_n$. This means that, as shown in \cite{Iorgulescu1998}, $\MV_n$ can be identified with a (proper, for $n \geq 5$) subvariety of the variety $\LM_n$ of $n$-valued \emph{\L ukasiewicz-Moisil algebras}, generated by the $n$-element chain with $\neg$ defined as for $\lucas_n$ and all $\tau_d$ as fundamental operations (see \cite{Boicescu1991} for an overview of \L ukasiewicz-Moisil algebras). 

The unary terms $\tau_d$ can be very useful, for example, they are important in the algebraic study of modal extensions of \L ukasiewicz finitely-valued logic \cite{Bou2011, HansoulTeheux2013}. Notably, as shown in \cite[pp. 344--345]{Ostermann1988}, only the operations $\odot$ and $\oplus$ are required to define these unary terms in $\lucas_n$. Thus, they will still be available in our study of finite positive $\MV$-chains (also see Lemma~\ref{lem:taus_definable}).                   
\subsection{Natural dualities}\label{subsec:NaturalDualities} The theory of natural dualities provides a common framework
to develop dual equivalences between quasi-varieties of algebras and structured Stone spaces. In particular, the theory encompasses and generalizes Stone duality for Boolean algebras and Priestley duality for distributive lattices. In this subsection, we give a selective overview of this theory. For more information, we refer the reader to the book \cite{ClarkDavey1998}, which we often cite throughout this paper.

Let $\alg{M}$ be a finite algebra (with underlying set $M$) and let $\var{A} = \qvariety{\alg{M}}$ be the quasi-variety it generates. An \emph{alter ego} of $\alg{M}$ is a discrete topological structure (also with underlying set $M$) of the form 
\[
\utilde{\alg{M}} = \langle M, \mathcal{G}, \mathcal{H}, \mathcal{R}, \mathcal{T}_{\mathrm{dis}} \rangle,
\]    
where $\mathcal{G}$ is a collection of (total) homomorphisms $\alg{M}^n \to \alg{M}$ (possibly nullary, which corresponds to \emph{constants}), $\mathcal{H}$ is a collection of \emph{partial homomorphisms}, that is, homomorphisms from a subalgebra of $\alg{M}^n$ to $\alg{M}$ and $\mathcal{R}$ is a collection of \emph{algebraic relations}, that is, subalgebras $\alg{R}\subseteq \alg{M}^n$. Lastly, $\mathcal{T}_{\mathrm{dis}}$ is the discrete topology on $M$. 

The \emph{topological quasi-variety} $\var{X} = \tqvariety{\utilde{\alg{M}}}$ generated by $\utilde{\alg{M}}$ consists of structured Stone spaces (recall that a \emph{Stone space} is a topological space $(X,\mathcal{T})$ which is compact, Hausdorff and totally disconnected)
\[
\alg{X} = \langle X, \mathcal{G}^\alg{X}, \mathcal{H}^\alg{X}, \mathcal{R}^\alg{X}, \mathcal{T} \rangle,
\]
of the same type as $\utilde{\alg{M}}$ which are isomorphic to a closed substructure of a non-empty product of $\utilde{\alg{M}}$. The category $\mathcal{X}$ with structure-preserving continuous maps as morphisms is often described using the Preservation Theorem~\cite[Theorem 1.4.3]{ClarkDavey1998} and the Separation Theorem~\cite[Theorem 1.4.3]{ClarkDavey1998}. 

By the Preduality Theorem~\cite[Theorem 1.5.2]{ClarkDavey1998}, there exists a dual adjunction between $\var{A}$ and $\var{X}$ given by the contravariant hom-functors $\func{D}\colon \var{A} \to \var{X}$ and $\func{E}\colon \var{X} \to \var{A}$ defined by 
\[
\func{D}(\alg{A}) = \var{A}(\alg{A}, \alg{M}) \text{ and } \func{E}(\alg{X}) = \var{X}(\alg{X}, \utilde{\alg{M}}) 
\] 
for all $\alg{A} \in \var{A}$ and $\alg{X}\in \var{X}$. The natural transformations $e\colon 1_\var{A} \to \func{ED}$ and $\varepsilon\colon 1_{\var{X}} \to \func{DE}$ corresponding to this adjunction are given by evaluations
\begin{align*}
e_\alg{A}(a)(u)=u(a) \qquad & \text{for all } \alg{A}\in \var{A}, u \in \func{D}(\alg{A}) \text{ and } a \in A,\\
\varepsilon_\alg{X}(x)(\alpha)=\alpha(x) \qquad & \text{for all } \alg{X}\in \var{X}, \alpha \in \func{E}(\alg{X}) \text{ and } x \in X.
\end{align*}
If $e$ is a natural isomorphism, we say that $\utilde{\alg{M}}$ \emph{yields a duality for $\var{A}$} (this is also known as \emph{dual representation}). If $\varepsilon$ is a natural isomorphism as well, we say that $\utilde{\alg{M}}$ yields a \emph{full} duality for $\var{A}$ (meaning that $\func{D}$ and $\func{E}$ establish a \emph{dual equivalence}). In fact, in this paper we exclusively deal with \emph{strong dualities} \cite[Chapter 3]{ClarkDavey1998}, which are full dualities with the additional property that $\utilde{\alg{M}}$ is injective in $\var{X}$. 

In particular, for lattice-based algebras, strong dualities can often be obtained via the NU Strong Duality Corollary \cite[Corollary 3.3.9]{ClarkDavey1998}.
\begin{corollary}\label{cor:NUStrongDuality} \cite{ClarkDavey1998}
Let $\alg{M}$ have a majority term, and let all subalgebras of $\alg{M}$ be subdirectly irreducible. Then 
\[
\utilde{\alg{M}} =  \langle M, K, P_1, \mathbb{S}(\alg{M}\times \alg{M}), \mathcal{T}_{\mathrm{dis}} \rangle,
\]
yields a strong duality on $\var{A}$, where
$K$ is the union of trivial (\emph{i.e.}, one-element) subalgebras of $\alg{M}$, the set $P_1$ consists of all unary partial homomorphisms $\alg{M} \to \alg{M}$ and $\mathbb{S}(\alg{M}\times \alg{M})$ consists of all binary algebraic operations. 
\end{corollary}

While this corollary narrows down the structure needed to obtain a strong duality, this $\utilde{\alg{M}}$ is usually still more complicated than it necessarily has to be. This is where (strong) \emph{entailment} comes into play. We say that another alter ego $\utilde{\alg{M}}' = \langle M, \mathcal{G}',\mathcal{H}',\mathcal{R}', \mathcal{T}_{\mathrm{dis}} \rangle$ \emph{strongly entails} $\utilde{\alg{M}}$ if whenever $\utilde{\alg{M}}$ yields a strong duality on $\var{A}$, the same is true for $\utilde{\alg{M}'}$. Similarly, we say that members of $\mathcal{G}' \cup \mathcal{H}' \cup \mathcal{R}'$ strongly entail members of $\mathcal{G} \cup \mathcal{H} \cup \mathcal{R}$. In the following, we give a list of admissible constructs for strong entailment relevant for this paper (see \cite[Chapter 9]{ClarkDavey1998} for a complete list of admissible constructs for entailment). 
\begin{enumerate}
\item Any set of relations strongly entails the full product $\alg{M}^2$, the diagonal $\Delta_\alg{M} = \{ (m,m) \mid m\in \alg{M} \}$ of $\alg{M}$ and the identity $\id_\alg{M}$ on $\alg{M}$.
\item Any binary relation $\alg{R}$ strongly entails its converse $\alg{R}^{-1} = \{ (b,a) \mid (a,b\in\alg{R})\}$ and $\pi_1(\alg{R} \cap \Delta_\alg{M})$.
\item Relations $\alg{S},\alg{R} \subseteq \alg{M}^n$ strongly entail their intersection $\alg{S} \cap \alg{R}$. 
\item Arbitrary relations $\alg{S}$ and $\alg{R}$ entail their product $\alg{S} \times \alg{R}$.  
\item $\utilde{\alg{M}}'$ strongly entails $\utilde{\alg{M}}$ if it is obtained from $\utilde{\alg{M}}$ by deleting a partial operation $h\in \mathcal{H}$ which has an extension in $\mathcal{G}$ and adding its domain to $\mathcal{R}$.   
\end{enumerate} 

We say that $\utilde{\alg{M}}$ yields an \emph{optimal} strong duality if $\mathcal{G} \cup \mathcal{H} \cup \mathcal{R}$ is not strongly entailed by any of its proper subsets.  

We illustrate the concepts introduced in this subsection by explaining how to obtain natural dualities for $\MV_n$ (these dualities have been explored in \cite{Niederkorn2001}). This example is a specific instance of the proof of the Semi-primal Strong Duality Theorem \cite[Theorem 3.3.14]{ClarkDavey1998}.

\begin{example}\label{exam:SemiPrimalStrongDuality}
Let $n\geq 1$. The discrete structure
\[
\utilde{\lucas}_n = \langle \{ 0, \tfrac{1}{n},\dots \tfrac{n-1}{n}, 1 \}, \mathbb{S}(\lucas_n), \mathcal{T}_{\mathrm{dis}} \rangle, 
\]
where members of $\mathbb{S}(\lucas_n)$ are understood as unary relations, yields a strong duality on $\MV_n$. 
\end{example}
\begin{proof} 
By Corollary~\ref{cor:NUStrongDuality}, the structure
\[
\langle \{ 0, \tfrac{1}{n},\dots \tfrac{n-1}{n}, 1 \}, K, P_1, \mathbb{S}(\lucas_n\times \lucas_n), \mathcal{T}_{\mathrm{dis}} \rangle
\]
yields a strong duality on $\MV_n$ (where $K$ is the union of one-element subalgebras and $P_1$ is the collection of all unary partial homomorphisms). Since $\lucas_n$ is based on a bounded lattice, it has no one-element subalgebras, therefore $K = \varnothing$. Furthermore, the only homomorphism $\lucas_k\to \lucas_n$ defined on a subalgebra $\lucas_k \subseteq \lucas_n$ is the natural embedding of $\lucas_k$. Using the strong entailment constructs (1) and (5) above, it can be replaced by its domain $\lucas_k \in \mathbb{S}(\lucas_n)$. Every subalgebra $\alg{R} \in \mathbb{S}(\lucas_n \times \lucas_n)$ is  simply a product of subalgebras of $\lucas_n$. Therefore, by (4) above, they are strongly entailed by $\mathbb{S}(\lucas_n)$ as well.            
\end{proof}
It follows from \cite[Theorem 9.2.6]{ClarkDavey1998} that modifying the structure from Example~\ref{exam:SemiPrimalStrongDuality} to only include the meet-irreducible members of the lattice $\mathbb{S}(\lucas_n)$ yields an \emph{optimal} strong duality (also see \cite[Theorem 8.3.2]{ClarkDavey1998}). 

In the next section, we aim to come up with a similarly simple natural duality for varieties generated by \emph{positive} $\MV$-chains.  
\section{Natural dualities for varieties generated by positive MV-chains}\label{sec:NatDualposMv}
In Subsection~\ref{subsec:PositiveMVChains}, we introduce the varieties $\PMV_n$ of positive $\MV_n$-algebras, generated by the positive $\MV_n$-chains $\PL_n$. We prove some basic facts about congruences and subalgebras of $\PL_n$ and show that the variety generated by $\PL_n$ coincides with the quasi-variety generated by $\PL_n$. In Subsection~\ref{subsec:NaturalDualitiesPosMV}, we develop our natural dualities for the varieties $\PMV_n$. In particular, to this end the systematic study of subalgebras of the order relation $\leq$ (which is itself a subalgebra of $\PL_n \times \PL_n$) plays an important role.     
\subsection{Positive MV-chains}\label{subsec:PositiveMVChains}
Following the recent paper \cite{Abbadini2022}, we use the term positive $\MV$-algebra to refer to a negation-free (and implication-free) subreduct of an $\MV$-algebra. In particular, we focus on finite positive $\MV$-chains defined as follows.      
\begin{definition}\label{def:PMV_n}
Let $n\geq 1$ be a natural number. The \emph{$(n+1)$-element positive $\MV$-chain} is given by 
\[
\PL_n = \langle \{ 0, \tfrac{1}{n},\dots \tfrac{n-1}{n}, 1 \}, \wedge , \vee, \odot, \oplus, 0, 1\rangle,
\]  
understood as a reduct of $\lucas_n$. We write $\PMV_n$ for the variety $\variety{\PL_n}$ generated by $\PL_n$, and we refer to members of $\PMV_n$ as \emph{positive $\MV_n$-algebras} or $\PMV_n$-algebras.
\end{definition} 
Our first result about $\PL_n$ is that its subalgebras are the same as the subalgebras of $\lucas_n$ and, therefore (recall Subsection~\ref{subsec:MV-Algebras}), the subalgebra-lattice $\mathbb{S}(\PL_n)$ is isomorphic to the bounded lattice of diviors of $n$. 
\begin{proposition}\label{prop:subalgebrasOfPLn}
The subalgebras of $\PL_n$ are exactly given by the subuniverses 
\[
\PL_k \cong \{ 0, \tfrac{\ell}{n}, \dots, \tfrac{(k-1)\ell}{n}, 1\},
\]
where $n = k\cdot\ell$.
\end{proposition}

\begin{proof}
Let $\alg{L} \subseteq \PL_n$ be an arbitrary subalgebra and let $\frac{\ell}{n}$ be the unique minimal element of $\alg{L}$ which is not zero. If $\ell = n$, then $\alg{L} = \PL_1$ holds, so assume $\ell < n$. Note that this implies $\frac{\ell}{n} \leq \frac{1}{2}$, since otherwise $\frac{\ell}{n}\odot \frac{\ell}{n}$ would be an element of $\alg{L}$ greater than zero but strictly smaller than $\frac{\ell}{n}$, contradicting our choice of $\ell$. Furthermore, $\ell$ needs to be a divisor of $n$, since otherwise we can find natural numbers $x \geq 1$ and $0 < r < \ell$ with $n = x\ell + r$. But then $\frac{x\ell}{n} \odot \frac{\ell}{n} = \frac{r}{n}$ is a member of $\alg{L}$ above zero but strictly below $\frac{\ell}{n}$, again contradicting our choice of $\ell$. Thus we showed that $\ell$ divides $n$ and therefore, by closure of $\alg{L}$ under $\oplus$, we showed that $\PL_k$ as in the proposition is contained in $\alg{L}$. 

Suppose towards contradiction that there is some $\frac{s}{n} \in \alg{L}{\setminus}\PL_k$. Then $\ell < s$ holds by the above assumption and we can find natural numbers $k > x > 1$ and $0 < r < \ell$ such that $s = x\ell + r$. This is equivalent to 
\[
r + n = s - x\ell + n = s + (k-x)\ell.
\]
Therefore, we conclude that $\frac{r}{n} = \frac{s}{n} \odot \frac{(k-x)\ell}{n}$ is in $\alg{L}$, once more contradicting minimality in our choice of $\ell$.              
\end{proof}

As noted at the end of Subsection~\ref{subsec:MV-Algebras}, the unary operations $\tau_d$ can be defined from $\odot$ and $\oplus$ alone. This fact will be of high importance in many proofs later on. 

\begin{lemma}\label{lem:taus_definable} \cite{Ostermann1988}
For every $d \in \PL_n$, the unary operation $\tau_d\colon \text{\L}_n \to \text{\L}_n$ given by 
\[
\tau_d(x) = 
\begin{cases}
1 & \text{ if } d \leq x, \\
0 & \text { otherwise } 
\end{cases}
\]
is term-definable in $\PL_n$.
\end{lemma}    

\begin{remark}\label{rem:OtherAlgebrasWithTaus}
While we chose to exclusively focus on $\PL_n$ in this paper, all results up until Lemma~\ref{lem:minimalrelation} actually hold for \emph{every} finite algebra $\alg{D}$ which has a bounded-lattice reduct and in which $\tau_d$ defined as above is term-definable in $\alg{D}$ for every $d\in \alg{D}$. In particular, this encompasses the negation-free reducts of the finite \emph{\L ukasiewicz-Moisil chains} (see, e.g., \cite{Boicescu1991}).
\end{remark}

Our first goal is to show that the variety $\PMV_n$ coincides with the quasi-variety $\qvariety{\PL_n}$ generated by $\PL_n$. For this, we essentially only have to show the following. 

\begin{lemma}\label{lem:subalgebrasSimple}
Every subalgebra $\PL_k \subseteq \PL_n$ (including $\PL_n$ itself) is simple. 
\end{lemma} 

\begin{proof}
Let $\theta$ be a congruence relation on $\alg{S}$ and let $c,d\in \PL_k$ be distinct elements with $(c,d)\in \theta$. We show that $\theta$ is the trivial congruence identifying all members of $\PL_k$. Without loss of generality, we assume $c < d$. 
Since $\tau_d$ from Lemma~\ref{lem:taus_definable} is term-definable in $\PL_n$, we have $(0,1) = (\tau_d(c), \tau_d(d)) \in \theta$ and $(1,0)\in \theta$ by symmetry. Now, for arbitrary $x,y\in \PL_k$, we have 
\[
(x,y) = \big((1,0) \wedge (x,x)\big) \vee \big((0,1)\wedge (y,y)\big) \in \theta,
\]
which implies $\theta = \PL_k^2$. 
\end{proof}

Since $\PMV_n$ is congruence distributive (because $\PL_n$ is lattice-based and thus has a majority term), a standard application of Jónsson's Lemma~\cite{Jonsson1967} yields the following (see, \emph{e.g.}, \cite[Theorem 1.3.6]{ClarkDavey1998}).

\begin{corollary}\label{cor:HSP=ISP}
$\PMV_n = \qvariety{\PL_n}$. 
\end{corollary}
This allows us to study the variety $\PMV_n$ via the theory of natural dualities in what follows.   
\subsection{The natural dualities}\label{subsec:NaturalDualitiesPosMV}
This subsection is dedicated to finding a simple alter-ego $\utilde{\PL}_n$ of $\PL_n$ which yields a `useful' \cite[Chapter 6]{ClarkDavey1998} strong duality on $\PMV_n$. Since $\PL_n$ has a bounded lattice reduct, it has a majority term and no trivial subalgebras. Furthermore, by Lemma~\ref{lem:subalgebrasSimple} we know that every subalgebra of $\PL_n$ is subdirectly irreducible. Therefore, we may use Corollary \ref{cor:NUStrongDuality} (\emph{i.e.}, the NU Strong Duality Corollary \cite[Corollary 3.3.9]{ClarkDavey1998}) as our starting point. This states that  
\begin{equation}\label{eq:1}
\langle \{ 0, \tfrac{1}{n},\dots, \tfrac{n-1}{n},1 \}, P_1, \mathbb{S}(\PL_n\times \PL_n), \mathcal{T}_{\mathrm{dis}} \rangle,
\end{equation}
yields a strong duality for $\PMV_n$, where $P_1$ is the set of all unary partial homomorphisms $\PL_n\to \PL_n$. In the following we show that, as for $\lucas_n$, the only partial homomorphisms of this kind are the identities of subalgebras of $\PL_n$.   

\begin{lemma}\label{lem:partialHomomorphismsAreIdentities}
Let $\PL_k\subseteq \PL_n$ be a subalgebra. Then the only homomorphism $\PL_k \to \PL_n$ is the identity on $\PL_k$ followed by inclusion. 
\end{lemma}

\begin{proof}
Let $h\colon \PL_k \to \alg{D}$ be a homomorphism. Suppose the are some $s\in \PL_k$ and $d \in \PL_n$ such that $h(s) = d$ and $s \neq d$. Recall that $\tau_d$ and $\tau_s$ from Lemma~\ref{lem:taus_definable} are term-definable and thus preserved by $h$.   
If $s < d$ then $1 = \tau_d(h(s)) = h(\tau_d(s)) = h(0) = 0$ yields a contradiction.
If $r < s$ then $1 = h(\tau_{s}(s)) = \tau_s(h(s)) = \tau_s(d) = 0$ also yields a contradiction. Thus no such elements $s$ and $d$ can exist and we showed that $h(s) = s$  holds for all $s\in \PL_k$. 
\end{proof}

Therefore, as in Example~\ref{exam:SemiPrimalStrongDuality}, the collection $P_1$ of unary partial homomorphisms is strongly entailed by the collection of unary algebraic relations $\mathbb{S}(\PL_n)$. Now we take a closer look at the binary algebraic relations in $\mathbb{S}(\PL_n\times \PL_n)$. Contrary to $\lucas_n$, the algebra $\PL_n\times \PL_n$ has subalgebras which are not direct products of subalgebras of $\PL_n$. For example, since all operations of $\PL_n$ are order-preserving, the relation $\leq$ and its converse $\geq$ are clearly subalgebras of $\PL_n\times \PL_n$. In the following, we show that every other subalgebra of $\PL_n$ which is not a direct product of subalgebras is contained in one of those.  

\begin{lemma}\label{lem:SubalgebrasOforder}
Every subalgebra $\alg{R}\subseteq \PL_n\times \PL_n$ which is not a direct product of subalgebras of $\PL_n$ is a subalgebra of $\leq$ or of $\geq$.   
\end{lemma}   

\begin{proof}
Suppose that $\alg{R}$ is neither a subset of $\leq$ nor of $\geq$. We show that this implies that $\alg{R}$ is a direct product of subalgebras of $\PL_n$. Since $\alg{R}$ is not a subset of $\leq$, there is $(d_1,c_1)\in \alg{R}$ with $d_1 > c_1$. Similarly, there is $(c_2,d_2) \in \alg{R}$ with $c_2 < d_2$. This implies that $(1,0) = \tau_{d_1}(d_1,c_1)$ and $(0,1) = \tau_{d_2}(c_2,d_2)$ are both in $\alg{R}$. As in the proof of Lemma~\ref{lem:subalgebrasSimple}, with this we can show that $\alg{R}$ is the full direct product of its two projections $\pr_1(\alg{R})$ and $\pr_2(\alg{R})$.
\end{proof}

Since every binary relation strongly entails its converse and all products of subalgebras of $\PL_n$ are strongly entailed by $\mathbb{S}(\PL_n)$, it follows that the structure  
\begin{equation}\label{eq:2}
\langle \{ 0, \tfrac{1}{n},\dots, \tfrac{n-1}{n},1 \}, \mathbb{S}(\PL_n) \cup \mathbb{S}(\leq), \mathcal{T}_{\mathrm{dis}} \rangle
\end{equation}
yields a strong duality for $\PMV_n$, since it strongly entails the structure from Equation~(\ref{eq:1}). 

While the structure given in Equation~(\ref{eq:2}) is already much simpler than that in Equation~(\ref{eq:1}), it is still far from optimal. Therefore, we keep on studying $\mathbb{S}(\leq)$ in order to further simplify this alter ego.   

A somewhat special role is played by the subalgebra of the order $\lhd \in \mathbb{S}(\leq)$ given by 
\[
\lhd = \{ (x,y) \mid x = 0 \text{ or } y = 1 \}.
\] 
It is easy to see that this is a subalgebra since $0 \wedge x = 0 \odot x = 0$ and $1 \vee x = 1 \oplus x = 1$ for all $x\in \PL_n$. Unfortunately, except for the case $n = 2$, this is not the only non-diagonal proper subalgebra of the order relation. However, it is minimal among those subalgebras in the following sense.

\begin{lemma}\label{lem:minimalrelation}
Let $\alg{R}\subseteq \PL_n\times \PL_n$ be a subalgebra of the order $\leq$, which is not the diagonal of a subalgebra of $\PL_n$, and $\alg{S} = \pr_1(\alg{R}) \times \pr_2(\alg{R})$. Then $\lhd{\mid}_{\alg{S}} \subseteq \alg{R} \subseteq {\leq}{\mid}_{\alg{S}}$. 
\end{lemma}  

\begin{proof}
Since $\alg{R}$ is not a diagonal, there is a pair $(x,y)\in \alg{R}$ with $x\neq y$, implying $x < y$. Therefore, $\tau_y(x,y) = (0,1) \in \alg{R}$ as well. Now, for any $(x',y') \in \alg{R}$ we find that 
\[
(0, y') = (x', y') \wedge (0,1) \text{ and } (x',1) = (x',y') \vee (0,1)
\]
are also members of $\alg{R}$, finishing the proof. 
\end{proof}       

Since diagonals of subalgebras are strongly entailed by $\mathbb{S}(\PL_n)$ already, we only need to consider subalgebras in-between (restrictions of) $\lhd$ and $\leq$. In order to describe these subalgebras, the following `closure' downwards in the first and upwards in the second component will be crucial.   

\begin{definition}\label{def:down/upClosure}
Let $\alg{S} = \PL_k \times \PL_{k'}$ be a product of subalgebras of $\PL_n$ (recall Proposition~\ref{prop:subalgebrasOfPLn}). Let $(x,y) \in \alg{S}$ with $x \leq y$. We denote by $C_{(x,y),\alg{S}}$ the following subset of $\alg{S}$ and $\leq$.
\[
C_{(x,y),\alg{S}} = \{ (x',y') \in \alg{S} \mid x'\leq x \text{ and } y \leq y'\}.
\]
If $\alg{S} = \PL_n\times \PL_n$, we simply use $C_{(x,y)}$ instead of $C_{(x,y),\alg{S}}$.  
\end{definition}

For example, Figure~\ref{fig:1} depicts the subsets $C_{(\frac{2}{6},\frac{3}{6})}$ and $C_{(\frac{2}{6},\frac{3}{6}),\alg{S}}$ for $\alg{S} = \PL_3\times \PL_2$ as subsets of $\PL_6\times \PL_6$.  

\begin{figure}[ht]
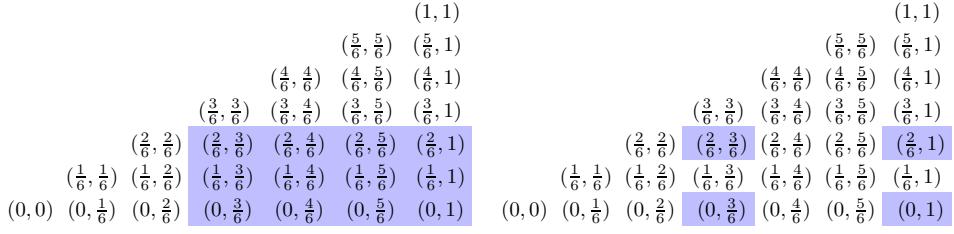

\begin{adjustbox}{width=\columnwidth,center}
\begin{tabular}{c c c}
\begin{tabular}{c c c c c c c}
&  &  &  &  &  &  $(1,1)$ \\

&  &  &  &  & $(\tfrac{5}{6},\tfrac{5}{6})$ &  $(\tfrac{5}{6},1)$ \\

&  &  &  & $(\tfrac{4}{6},\tfrac{4}{6})$ & $(\tfrac{4}{6},\tfrac{5}{6})$ &  $(\tfrac{4}{6},1)$ \\

&  & & $(\tfrac{3}{6},\tfrac{3}{6})$ & $(\tfrac{3}{6},\tfrac{4}{6})$ & $(\tfrac{3}{6},\tfrac{5}{6})$ &  $(\tfrac{3}{6},1)$ \\

& & $(\tfrac{2}{6},\tfrac{2}{6})$ & \cellcolor{blue!25} $(\tfrac{2}{6},\tfrac{3}{6})$ & \cellcolor{blue!25} $(\tfrac{2}{6},\tfrac{4}{6})$ & \cellcolor{blue!25} $(\tfrac{2}{6},\tfrac{5}{6})$ & \cellcolor{blue!25} $(\tfrac{2}{6},1)$ \\

& $(\tfrac{1}{6},\tfrac{1}{6})$ & $(\tfrac{1}{6},\tfrac{2}{6})$ & \cellcolor{blue!25} $(\tfrac{1}{6},\tfrac{3}{6})$ & \cellcolor{blue!25} $(\tfrac{1}{6},\tfrac{4}{6})$ & \cellcolor{blue!25} $(\tfrac{1}{6},\tfrac{5}{6})$ & \cellcolor{blue!25} $(\tfrac{1}{6},1)$ \\
 
 $(0,0)$ &  $(0,\tfrac{1}{6})$ &  $(0,\tfrac{2}{6})$ & \cellcolor{blue!25} $(0,\tfrac{3}{6})$ & \cellcolor{blue!25} $(0,\tfrac{4}{6})$ & \cellcolor{blue!25} $(0,\tfrac{5}{6})$ & \cellcolor{blue!25} $(0,1)$ \\ 
\end{tabular}
& &
\begin{tabular}{c c c c c c c}
&  &  &  &  &  &  $(1,1)$ \\

&  &  &  &  & $(\tfrac{5}{6},\tfrac{5}{6})$ &  $(\tfrac{5}{6},1)$ \\

&  &  &  & $(\tfrac{4}{6},\tfrac{4}{6})$ & $(\tfrac{4}{6},\tfrac{5}{6})$ &  $(\tfrac{4}{6},1)$ \\

&  & & $(\tfrac{3}{6},\tfrac{3}{6})$ & $(\tfrac{3}{6},\tfrac{4}{6})$ & $(\tfrac{3}{6},\tfrac{5}{6})$ &  $(\tfrac{3}{6},1)$ \\

& & $(\tfrac{2}{6},\tfrac{2}{6})$ & \cellcolor{blue!25} $(\tfrac{2}{6},\tfrac{3}{6})$ & $(\tfrac{2}{6},\tfrac{4}{6})$ & $(\tfrac{2}{6},\tfrac{5}{6})$ & \cellcolor{blue!25} $(\tfrac{2}{6},1)$ \\

& $(\tfrac{1}{6},\tfrac{1}{6})$ & $(\tfrac{1}{6},\tfrac{2}{6})$ & $(\tfrac{1}{6},\tfrac{3}{6})$ & $(\tfrac{1}{6},\tfrac{4}{6})$ & $(\tfrac{1}{6},\tfrac{5}{6})$ &  $(\tfrac{1}{6},1)$ \\
 
 $(0,0)$ &  $(0,\tfrac{1}{6})$ &  $(0,\tfrac{2}{6})$ & \cellcolor{blue!25} $(0,\tfrac{3}{6})$ & $(0,\tfrac{4}{6})$ &  $(0,\tfrac{5}{6})$ & \cellcolor{blue!25} $(0,1)$ \\ 
\end{tabular} \\
\end{tabular}
\end{adjustbox}
\caption{The sets $C_{(\frac{2}{6},\frac{3}{6})}$ and $C_{(\frac{2}{6},\frac{3}{6}),\PL_3 \times \PL_2}$ in the case $n = 6$.}
\label{fig:1}
\end{figure}
In the next lemma, we show that non-diagonal subalgebras of the order are closed under these subsets in the following sense.  
\begin{lemma}\label{lem:SubalgebraClosedUnderRectangles}
Let $\alg{R}\subseteq \PL_n\times \PL_n$ be a subalgebra of the order $\leq$ which is not the diagonal of a subalgebra of $\PL_n$, and $\alg{S} = \pr_1(\alg{R}) \times \pr_2(\alg{R})$. If $(x,y) \in \alg{R}$, then $C_{(x,y),\alg{S}} \subseteq \alg{R}$ as well.  
\end{lemma}

\begin{proof}
By Lemma~\ref{lem:minimalrelation} we know that $\lhd{\mid}_{\alg{S}} \subseteq \alg{R}$. Now let $(x,y) \in \alg{R}$, and say $(x',y')\in \alg{S}$ satisfies $x' \leq x$ and $y \leq y'$. Then $(x',y) = (x,y) \wedge (x',1)$ is in $\alg{R}$ and, thus, $(x',y') = (x',y) \vee (0,y')$ is also in $\alg{R}$. 
\end{proof}

Therefore, clearly every $\alg{R}$ as in the above lemma is a union of sets of the form $C_{(x,y),\alg{S}}$. However, not all unions of sets of this form necessarily yield subalgebras. In the following, we identify exactly those unions which \emph{do} give rise to subalgebras of $\PL_n\times\PL_n$. 

\begin{proposition}\label{prop:subalgebrasAsUnionsOfRectangles}
Let $\alg{S} = \PL_k \times \PL_{k'}$ be a product of subalgebras of $\PL_n$. 
\begin{enumerate}
\item Let $\alg{R} \subseteq \PL_n \times \PL_n$ be a subalgebra of $\leq$, which is not the diagonal of a subalgebra of $\PL_n$, with $\pr_1(\alg{R})\times \pr_2(\alg{R}) = \alg{S}$. Then $\alg{R}$ can be expressed as
\[
\alg{R} = \bigcup_{i = 0}^{k} C_{(\frac{i}{k},y_i),\alg{S}}
\]
where $y_i$ is the minimal element of $\PL_{k'}$ with $(\frac{i}{k}, y_i) \in \alg{R}$ (in particular, $y_0 = 0$ and $y_k = 1$).
\item Let $y_0, \dots, y_k$ be an increasing sequence of elements of $\PL_{k'}$ with $y_0 = 0, y_k = 1$ and $\frac{i}{k} \leq y_i$ for all $i = 1,\dots, k-1$. Then 
\[
\alg{R} = \bigcup_{i = 0}^{k} C_{(\frac{i}{k},y_i),\alg{S}}
\]
is a subalgebra of $\alg{S}$ if and only if the conditions
\[
(\tfrac{i}{k}, y_i) \odot (\tfrac{j}{k}, y_j) \in \alg{R} \text{ and } (\tfrac{i}{k}, y_i) \oplus (\tfrac{j}{k}, y_j) \in \alg{R}
\]    
hold for all $i,j \in \{ 1, \dots, k-1 \}$. 
\end{enumerate}
\end{proposition}

\begin{proof}
(1): By Lemma~\ref{lem:SubalgebraClosedUnderRectangles} we have $\bigcup_{i = 0}^{k} C_{(\frac{i}{k},y_i),\alg{S}} \subseteq \alg{R}$. Conversely, if $(\frac{i}{k}, y) \in \alg{R}$ then $y_i \leq y$ by minimality of $y_i$ and therefore $(\frac{i}{k}, y) \in C_{(\frac{i}{k},y_i),\alg{S}}$. 

(2): Clearly these conditions are necessary for $\alg{R}$ to be a subalgebra. We show that they are also sufficient. So, supposing these condition hold, we want to show that $\alg{R}$ is indeed a subalgebra. First note that $\lhd{\mid}_\alg{S} = C_{(0,0),\alg{S}} \cup C_{(1,1),\alg{S}} \subseteq \alg{R}$, in particular this implies that both constants $(0,0)$ and $(1,1)$ are contained in $\alg{R}$. Now let $(x,y)$ and $(x',y')$ be two elements of $\alg{R}$, say $(x,y) \in C_{(\frac{i}{k}, y_i),\alg{S}}$ and $(x',y') \in C_{(\frac{j}{k}, y_j),\alg{S}}$. Furthermore, without loss of generality we assume $i \leq j$. 

We first establish the closure under the lattice operations. To show closure under meets, we note that 
\[
(x,y) \wedge (x',y') = \begin{cases} 
(x,y) & \text{if } x \leq x', y\leq y', \\
(x',y') & \text{if } x' \leq x, y' \leq y, \\
(x',y) & \text {if } x' \leq x, y \leq y', \\
(x,y') & \text{if } x \leq x', y' \leq y.
\end{cases}
\] 
In the first two cases the meet is obviously still in $\alg{R}$. In the third case the two inequalities $x' \leq x \leq \frac{i}{k}$ and $y_i \leq y$ imply $(x',y) \in C_{(\frac{i}{k}, y_i),\alg{S}}$. In the fourth and final case the two inequalities $x \leq x' \leq \frac{j}{k}$ and $y_j \leq y'$ imply $(x,y')\in C_{(\frac{j}{k}, y_j),\alg{S}}$. Closure under joins is established analogously since 
\[
(x,y) \vee (x',y') = \begin{cases} 
(x,y) & \text{if } x' \leq x, y' \leq y, \\
(x',y') & \text{if } x \leq x', y \leq y', \\
(x',y) & \text {if } x \leq x', y' \leq y, \\
(x,y') & \text{if } x' \leq x, y \leq y'.
\end{cases}
\]
Note that in the third case we get $(x',y) \in C_{(\frac{j}{k}, y_j),\alg{S}}$ and in the fourth case we get $(x,y')\in C_{(\frac{i}{k}, y_i),\alg{S}}$.

Now let $\ast \in \{ \odot, \oplus \}$, and note that $(x,y) \in C_{(\frac{i}{k}, y_i),\alg{S}}$ and $(x',y') \in C_{(\frac{j}{k}, y_j),\alg{S}}$ together with monotonicity of $\ast$ imply 
\[
x \ast x' \leq \tfrac{i}{k} \ast \tfrac{j}{k} \text{ and } y_i \ast y_j \leq y \ast y'.
\]
However, by assumption we have $(\frac{i}{k}\ast \frac{j}{k}, y_i \ast y_j) \in \alg{R}$ say it is contained in $C_{(\frac{h}{k}, y_h),\alg{S}}$. Thus 
\[
x \ast x' \leq \tfrac{i}{k} \ast \tfrac{j}{k} \leq \tfrac{h}{k} \text{ and } y_h \leq  y_i \ast y_j \leq y \ast y'
\]
immediately implies that $(x,y) \ast (x',y')$ is also contained in $C_{(\frac{h}{k}, y_h),\alg{S}}$, finishing the proof.  
\end{proof}

For example, in Figure~\ref{fig:2}, on the left hand side the union 
\[
C_{(0,0)} \cup C_{(\frac{1}{6},\frac{2}{6})} \cup C_{(\frac{2}{6},\frac{3}{6})} \cup C_{(\frac{3}{6},\frac{5}{6})} \cup C_{(\frac{4}{6},1)} \cup C_{(\frac{5}{6},1)}\cup C_{(1,1)}
\] 
inside $\PL_6\times \PL_6$ is depicted. By Proposition~\ref{prop:subalgebrasAsUnionsOfRectangles}, we can easily confirm that this is a subalgebra by checking that the `corner elements' $(\frac{1}{6},\frac{2}{6}), (\frac{2}{6},\frac{3}{6})$ and $(\frac{3}{6},\frac{5}{6})$ are closed under the opertaions $\odot$ and $\oplus$. On the right hand side of Figure~\ref{fig:2}, the union 
\[
C_{(0,0)} \cup C_{(\frac{1}{6},\frac{2}{6})} \cup C_{(\frac{2}{6},\frac{3}{6})} \cup C_{(\frac{3}{6},1)} \cup C_{(\frac{4}{6},1)} \cup C_{(\frac{5}{6},1)}\cup C_{(1,1)}
\]  
is depicted. This is not a subalgebra because $(\frac{1}{6},\frac{2}{6}) \oplus (\frac{2}{6},\frac{3}{6}) = (\frac{3}{6}, \frac{5}{6})$ is not contained in this union.  
\begin{figure}[ht]
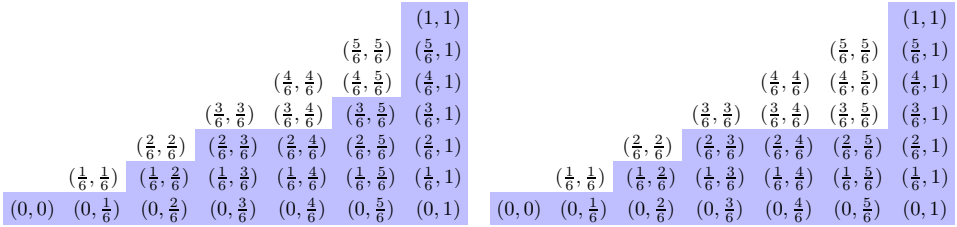

\begin{adjustbox}{width=\columnwidth,center}
\begin{tabular}{c c c}
\begin{tabular}{c c c c c c c}
&  &  &  &  &  & \cellcolor{blue!25} $(1,1)$ \\

&  &  &  &  & $(\tfrac{5}{6},\tfrac{5}{6})$ & \cellcolor{blue!25} $(\tfrac{5}{6},1)$ \\

&  &  &  & $(\tfrac{4}{6},\tfrac{4}{6})$ & $(\tfrac{4}{6},\tfrac{5}{6})$ & \cellcolor{blue!25} $(\tfrac{4}{6},1)$ \\

&  & & $(\tfrac{3}{6},\tfrac{3}{6})$ & $(\tfrac{3}{6},\tfrac{4}{6})$ & \cellcolor{blue!25} $(\tfrac{3}{6},\tfrac{5}{6})$ & \cellcolor{blue!25} $(\tfrac{3}{6},1)$ \\

& & $(\tfrac{2}{6},\tfrac{2}{6})$ & \cellcolor{blue!25} $(\tfrac{2}{6},\tfrac{3}{6})$ & \cellcolor{blue!25} $(\tfrac{2}{6},\tfrac{4}{6})$ & \cellcolor{blue!25} $(\tfrac{2}{6},\tfrac{5}{6})$ & \cellcolor{blue!25} $(\tfrac{2}{6},1)$ \\

& $(\tfrac{1}{6},\tfrac{1}{6})$ & \cellcolor{blue!25} $(\tfrac{1}{6},\tfrac{2}{6})$ & \cellcolor{blue!25} $(\tfrac{1}{6},\tfrac{3}{6})$ & \cellcolor{blue!25} $(\tfrac{1}{6},\tfrac{4}{6})$ & \cellcolor{blue!25} $(\tfrac{1}{6},\tfrac{5}{6})$ & \cellcolor{blue!25} $(\tfrac{1}{6},1)$ \\
 
 $\cellcolor{blue!25} (0,0)$ & \cellcolor{blue!25} $(0,\tfrac{1}{6})$ & \cellcolor{blue!25} $(0,\tfrac{2}{6})$ & \cellcolor{blue!25} $(0,\tfrac{3}{6})$ & \cellcolor{blue!25} $(0,\tfrac{4}{6})$ & \cellcolor{blue!25} $(0,\tfrac{5}{6})$ & \cellcolor{blue!25} $(0,1)$ \\ 
\end{tabular}
& &
\begin{tabular}{c c c c c c c}
&  &  &  &  &  & \cellcolor{blue!25} $(1,1)$ \\

&  &  &  &  & $(\tfrac{5}{6},\tfrac{5}{6})$ & \cellcolor{blue!25} $(\tfrac{5}{6},1)$ \\

&  &  &  & $(\tfrac{4}{6},\tfrac{4}{6})$ & $(\tfrac{4}{6},\tfrac{5}{6})$ & \cellcolor{blue!25} $(\tfrac{4}{6},1)$ \\

&  & & $(\tfrac{3}{6},\tfrac{3}{6})$ & $(\tfrac{3}{6},\tfrac{4}{6})$ & $(\tfrac{3}{6},\tfrac{5}{6})$ & \cellcolor{blue!25} $(\tfrac{3}{6},1)$ \\

& & $(\tfrac{2}{6},\tfrac{2}{6})$ & \cellcolor{blue!25} $(\tfrac{2}{6},\tfrac{3}{6})$ & \cellcolor{blue!25} $(\tfrac{2}{6},\tfrac{4}{6})$ & \cellcolor{blue!25} $(\tfrac{2}{6},\tfrac{5}{6})$ & \cellcolor{blue!25} $(\tfrac{2}{6},1)$ \\

& $(\tfrac{1}{6},\tfrac{1}{6})$ & \cellcolor{blue!25} $(\tfrac{1}{6},\tfrac{2}{6})$ & \cellcolor{blue!25} $(\tfrac{1}{6},\tfrac{3}{6})$ & \cellcolor{blue!25} $(\tfrac{1}{6},\tfrac{4}{6})$ & \cellcolor{blue!25} $(\tfrac{1}{6},\tfrac{5}{6})$ & \cellcolor{blue!25} $(\tfrac{1}{6},1)$ \\
 
 $\cellcolor{blue!25} (0,0)$ & \cellcolor{blue!25} $(0,\tfrac{1}{6})$ & \cellcolor{blue!25} $(0,\tfrac{2}{6})$ & \cellcolor{blue!25} $(0,\tfrac{3}{6})$ & \cellcolor{blue!25} $(0,\tfrac{4}{6})$ & \cellcolor{blue!25} $(0,\tfrac{5}{6})$ & \cellcolor{blue!25} $(0,1)$ \\ 
\end{tabular} \\
\end{tabular}
\end{adjustbox}
\caption{Only the subset on the left is a subalgebra of $\PL_6\times\PL_6$.}
\label{fig:2}
\end{figure}

Now that we have a good grasp on the subalgebras of $\PL_n \times \PL_n$, we aim to show that, ultimately, only subdirect products $\alg{R} \subseteq \PL_n \times \PL_n$ (meaning $\pr_1(\alg{R}) = \pr_2(\alg{R}) = \PL_n$) will be relevant for the natural duality. By Lemma~\ref{lem:minimalrelation}, this is equivalent to saying only the following relations will be relevant to the natural duality.  

\begin{definition}\label{def:SetOfRelevantRelations}
Let $\mathcal{S}_n \subseteq \mathbb{S}(\PL_n\times \PL_n)$ be the collection of all subalgebras $\alg{R} \subseteq \PL_n \times \PL_n$ which satisfy $\lhd  \subseteq \alg{R} \subseteq {\leq}$.  
\end{definition}

It is clear by definition that $\mathcal{S}_n$ is a bounded sublattice of $\mathbb{S}(\PL_n \times \PL_n)$ with lower bound $\lhd$ and upper bound $\leq$. 

In the next two (technical) lemmas, we show that the set of relations $\mathcal{S}_n$ strongly entails $\mathbb{S}(\leq)$ and $\mathbb{S}(\PL_n)$. The first lemma shows that relations $\alg{R} \in \mathbb{S}(\leq)$ with $\pr_1(\alg{R}) \times \pr_2(\alg{R}) \neq \PL_n \times \PL_n$ are strongly entailed by $\mathbb{S}(\PL_n)$ and $\mathcal{S}_n$.     

\begin{lemma}\label{lem:SubdirectProductsSuffice}
Let $\alg{R}\subseteq \PL_n\times \PL_n$ be a subalgebra of the order $\leq$, which is not the diagonal of a subalgebra of $\PL_n$, and let $\alg{S} = \pr_1(\alg{R}) \times \pr_2(\alg{R}) = \PL_k \times \PL_{k'}$ for some divisors $k, k'$ of $n$. Then there exists a subalgebra $\overline{\alg{R}} \in \mathcal{S}_n$ with $\alg{R} = \overline{\alg{R}} \cap \alg{S}$. 
\end{lemma} 
\begin{proof}
By Proposition~\ref{prop:subalgebrasAsUnionsOfRectangles}(1), we know that $\alg{R}$ can be expressed as union
\[
\alg{R} = \bigcup_{i = 0}^{k} C_{(\frac{i}{k},y_i),\alg{S}}
\]
where $y_i$ is the minimal element of $\PL_{k'}$ with $(\frac{i}{k}, y_i) \in \alg{R}$. Let $n = k\cdot\ell$. We define $\overline{\alg{R}}$ by 
\[
\overline{\alg{R}} = \bigcup_{j = 0}^n C_{(\frac{j}{n}, \hat{y}_j)}
\]
where we stipulate $\hat{y}_0 = 0$ and 
\[
\hat{y}_j = \begin{cases}
y_1 & \text{if } 1 \leq j \leq \ell, \\
y_2 & \text{if } \ell +1 \leq j \leq 2\ell, \\
\vdots & \vdots \\ 
y_{k-1} & \text{if } (k-2)\ell + 1 \leq j \leq (k-1)\ell, \\
y_k = 1 & \text{if } (k-1)\ell + 1 \leq j \leq n.   
\end{cases}
\] 
We show that $\overline{\alg{R}}$ is a subalgebra of $\PL_n \times \PL_n$ using Proposition~\ref{prop:subalgebrasAsUnionsOfRectangles}(2). That is, for any $j_1, j_2 \in \{ 1, \dots, n-1 \}$, we want to show that $(\frac{j_1}{n}, \hat{y}_{j_1}) \ast (\frac{j_2}{n}, \hat{y}_{j_2}) \in \overline{\alg{R}}$ holds for the $\MV$-operations $\ast \in \{ \odot, \oplus \}$. 

Let $i_1,i_2 \in \{ 1,\dots, k\}$ be the unique elements satisfying 
\[
(i_1 - 1)\ell < j_1 \leq  i_1\ell \text{ and } (i_2 - 1)\ell < j_2 \leq  i_2\ell,
\]
which by definition means $\hat{y}_{j_1} = y_{i_1}$ and $\hat{y}_{j_2} = y_{i_2}$. Since $\alg{R}$ is a subalgebra, we know that $(\frac{i_1}{k}, y_{i_1}) \ast (\frac{i_2}{k}, y_{i_2}) \in \alg{R}$, say it is in $C_{(\frac{h}{k}, y_h),\alg{S}}$. Now because $\frac{j_1}{n} \leq \frac{i_1\ell}{n} = \frac{i_1}{k}$ and similarly for $j_2, i_2$, we have 
\[ 
\frac{j_1}{n} \ast \frac{j_2}{n} \leq \frac{i_1}{k} \ast \frac{i_2}{k} \leq \frac{h}{k} = \frac{h\ell}{n} \]
and furthermore  
\[
y_h \leq y_{i_1} \ast y_{i_2} = \hat{y}_{j_1} \ast \hat{y}_{j_2}.
\]
Because $\hat{y}_{h\ell} = y_h$, this shows that $(\frac{j_1}{n}, \hat{y}_{j_1}) \ast (\frac{j_2}{n}, \hat{y}_{j_2}) \in C_{(\frac{h\ell}{n}, \hat{y}_{h\ell})} \subseteq \overline{\alg{R}}$, finishing the proof.          
\end{proof} 

Our second lemma shows that the collection $\mathbb{S}(\PL_n)$ is strongly entailed by $\mathcal{S}_n$. 
\begin{lemma}\label{lem:SubdirectProductsEntailSubalgebras}
For every $\PL_k \in \mathbb{S}(\PL_n)$, there exists a $\alg{R} \in \mathcal{S}_n$ such that $\alg{R} \cap \Delta_{\PL_n} = \Delta_{\PL_k}$ (where $\Delta_{\alg{A}}$ denotes the diagonal of the corresponding algebra $\alg{A}$).   
\end{lemma} 
\begin{proof}
Let $n = k \cdot \ell$ and $\PL_k$ be given as in Proposition~\ref{prop:subalgebrasOfPLn}. We define $\alg{R}$ by 
\[
\alg{R} = \bigcup_{i = 0}^n C_{(\frac{i}{n}, y_i)}
\]
where we stipulate $y_0 = 0$ and
\[
y_i = \begin{cases}
\frac{\ell}{n} & \text{if } 1 \leq j \leq \ell, \\
\frac{2\ell}{n} & \text{if } \ell +1 \leq j \leq 2\ell, \\
\vdots & \vdots \\ 
\frac{(k-1)\ell}{n} & \text{if } (k-2)\ell + 1 \leq j \leq (k-1)\ell, \\
1 & \text{if } (k-1)\ell + 1 \leq j \leq n.   
\end{cases}
\]  
By definition it is clear that $\alg{R} \cap \Delta_{\PL_n} = \Delta_{\PL_k}$, so we only have to show that $\alg{R}$ is a subalgebra. For this, we again use Proposition\ref{prop:subalgebrasAsUnionsOfRectangles}(2). Let $i_1, i_2 \in \{ 1, \dots, n-1 \}$ and let $j_1, j_2$ be the unique elements of $\{ 1, \dots, k\}$ with 
\[
(j_1 - 1)\ell < i_1 \leq  j_1\ell \text{ and } (j_2 - 1)\ell < i_2 \leq  j_2\ell,
\]    
which means that $y_{i_1} = \frac{j_1 \ell}{n}$ and $y_{i_2} = \frac{j_2 \ell}{n}$. Furthermore, let $\frac{j_1\ell}{n}\ast \frac{j_2\ell}{n} = \frac{h \ell}{n}$ (note that such an $h$ exists because $\PL_k$ is a subalgebra). Then 
\[
\frac{i_1}{n} \ast \frac{i_2}{n} \leq \frac{j_1 \ell}{n} \ast \frac{j_2 \ell}{n} = \frac{h\ell}{n}
\] 
implies 
\[
(\tfrac{i_1}{n}, y_{i_2}) \ast (\tfrac{i_2}{n}, y_{i_2}) \in C_{(\frac{h\ell}{n}, \frac{h\ell}{n})} \subseteq \alg{R}, 
\] 
which finishes the proof.
\end{proof}

With these two lemmas at hand, we are ready to state and easily prove the main theorem of this section.      

\begin{theorem}\label{thm:NaturalDualityPositiveMVn}
Let $n\geq 1$. The discrete relational structure 
\[
\utilde{\PL}_n = \langle \{ 0, \tfrac{1}{n},\dots, \tfrac{n-1}{n}, 1\}, \mathcal{S}_n, \mathcal{T}_{\mathrm{dis}} \rangle
\]
yields a strong duality for $\PMV_n$. 
\end{theorem}

\begin{proof}
By the discussion after Lemma~\ref{lem:SubalgebrasOforder}, we know that the structure given in Equation~(\ref{eq:2}), that is $\langle \{ 0, \tfrac{1}{n},\dots, \tfrac{n-1}{n},1 \}, \mathbb{S}(\PL_n) \cup \mathbb{S}(\leq), \mathcal{T}_{\mathrm{dis}} \rangle$, yields a strong duality for $\PMV_n$. By Lemma~\ref{lem:SubdirectProductsSuffice}, we know that every $\alg{R} \in \mathbb{S}(\leq)$ is an intersection of (and thus strongly entailed by) a product of subalgebras of $\PL_n$ and a relation from $\mathcal{S}_n$. By Lemma~\ref{lem:SubdirectProductsEntailSubalgebras} subalgebras of $\PL_n$ are strongly entailed by $\mathcal{S}_n$ as well.     
\end{proof}    

In light of Proposition~\ref{prop:subalgebrasAsUnionsOfRectangles}, it is fairly straightforward to find the lattice $\mathcal{S}_n$ in a systematic way. Indeed, in Appendix~\ref{appendix} we provide an easy algorithm to compute this lattice. Also note that, to obtain an \emph{optimal duality} (see \cite[Chapters 8 and 9]{ClarkDavey1998}), we could simplify the above structure further by only including meet-irreducible elements of $\mathcal{S}_n$ (this follows from \cite[Theorem 9.2.6]{ClarkDavey1998}). However, since it won't make a significant difference in this paper, we keep working with the alter ego from Theorem~\ref{thm:NaturalDualityPositiveMVn}.       

\begin{definition}\label{def:DualCategories}
For all $n\geq 1$, let $\var{X}_n$ be the topological quasi-variety $\tqvariety{\utilde{\PL}_n}$ generated by the structure from Theorem~\ref{thm:NaturalDualityPositiveMVn}. Furthermore, let $\func{D}_n \colon \PMV_n \to \var{X}_n$ and $\func{E}_n\colon \var{X}_n \to \PMV_n$ be the (hom-)functors establishing the corresponding dual equivalence. 
\end{definition}
Note that these dualities can be seen as many-valued generalizations of Priestley duality, which is recovered in the case where $n=1$.  

In the following, we collect some consequences of Theorem~\ref{thm:NaturalDualityPositiveMVn} which can be immediately derived from the general theory of natural dualities.
\begin{corollary}\label{cor:ConsequencesOfNatDuality}
The categories $\PMV_n$ and $\var{X}_n$ have the following properties.  
\begin{enumerate}
\item $\PL_n$ is injective in $\PMV_n$ and $\utilde{\PL}_n$ is injective in $\var{X}_n$.    
\item The injectives in $\PMV_n$ are exactly the Boolean powers $\PL_n[\alg{B}]$, where $\alg{B}$ is a non-trivial complete Boolean algebra.
\item $\PMV_n$ has the amalgamation property. 
\item A morphism $\varphi \colon \alg{X}_1 \to \alg{X}_2$ in $\var{X}_n$ is an embedding (a surjection) if and only if $\func{E}_n (\varphi)$ is a surjection (an embedding). A homomorphism $h \colon \alg{A}_1 \to \alg{A}_2$ in $\PMV_n$ is an embedding (a surjection) if and only if $\func{D}_n (h)$ is a surjection (an embedding).
\item The congruence lattice of $\alg{A} \in \PMV_n$ is dually isomorphic to the lattice of closed substructures of $\func{D}_n(\alg{A})$. 
\item Coproducts in $\var{X}_n$ are given by direct union (\emph{i.e.}, the duality is \emph{logarithmic}).           
\end{enumerate}
\end{corollary}
\begin{proof}
The second part of statement (1) follows from the definition of strong duality, the first part follows from \cite[Lemma 3.2.10]{ClarkDavey1998} and the fact that $\utilde{\PL}_n$ is a total structure. Statement (2) follows from \cite[Theorem 5.5.15]{ClarkDavey1998} because all relations $\alg{R} \in \mathcal{S}_n$ avoid binary products. Statement (3) follows from \cite[Lemma 5.3.4]{ClarkDavey1998}. Statement (4) follows from (1) and \cite[Lemmas 3.2.6 and 3.2.8]{ClarkDavey1998}. Statement (5) follows from \cite[Theorem 3.2.1]{ClarkDavey1998}. Lastly, statement (6) follows from \cite[Theorem 6.3.3]{ClarkDavey1998}.   
\end{proof}    

This corollary already demonstrates how useful these dualities are. In the next section, we investigate it further to derive more results about the varieties $\PMV_n$. 

\section{Further explorations of the dualities}\label{sec:FurtherExplorations} 
In this section, we delve deeper into various aspects of the natural dualities established in the previous section. In Subsection~\ref{subsec:ThreeElementChain}, we give a concrete axiomatization of the category $\var{X}_2$ dual to the variety $\PMV_2$ generated by the three-element positive $\MV$-chain. In Subsection~\ref{subsec:RelationshipPriestley}, we explore the relationship between the natural duality for $\PMV_n$ and Priestley duality. Lastly, in Subsection~\ref{subsec:A-E-ClosedAlgebras}, we give complete characterizations of algebraically and existentially closed algebras in $\PMV_n$.      
\subsection{The dual category for the three-element positive MV-chain.}\label{subsec:ThreeElementChain}
Among the finitely-valued \L ukasiewicz logics, arguably the most popular is the three-valued logic corresponding to the variety $\MV_2$ generated by the three-element $\MV$-chain $\lucas_2$. In this section, we focus on the variety $\PMV_2$ generated by the positive three-element $\MV$-chain $\PL_2$. More specifically, we provide an explicit description of the category $\var{X}_2$ dual to $\PMV_2$.   
\begin{theorem}\label{thm:AxiomatizationPMV2}
A structured Stone space $\alg{X} = \langle X, \lhd^\alg{X}, \leq^\alg{X},\mathcal{T}\rangle$ with binary relations $\lhd^\alg{X}$ and $\leq^\alg{X}$ closed in $X^2$ is a member of $\var{X}_2$ if and only if it satisfies the following axioms. 
\begin{enumerate}[(a)]
\item $x \lhd^\alg{X} y \Rightarrow x \leq^\alg{X} y$.
\item $(X,\leq^\alg{X}, \mathcal{T})$ is a Priestley space, that is, $\leq^\alg{X}$ is a partial order and 
if $x \not\leq^\alg{X} y$, then there exists a clopen upset $U$ containing $x$ but not $y$. 
\item If $x \ntriangleleft^\alg{X} y$ but $x \leq^\alg{X} y$, then there exist a clopen upset $U$ and a  clopen downset $D$ with the following properties 
\begin{itemize}
\item $x \notin D$ and $y\notin U$,  
\item For all $z,z' \in X$, if $z \lhd^\alg{X} z'$ then $z \in D$ or $z'\in U$. 
\end{itemize}   
\end{enumerate}
\end{theorem}
\begin{proof}
First we show that every member $\alg{X} = \langle X,\lhd^\alg{X},\leq^\alg{X},\mathcal{T} \rangle$ of $\var{X}_2$ satisfies (a)-(c). The formula (a) is quasi-atomic and holds in $\PL_n$, therefore, by the Preservation Theorem \cite[Theorem 1.4.3]{ClarkDavey1998}, it also holds for all members of $\var{X}_2$.  

To see condition (b) that $(X,\leq^\alg{X}, \mathcal{T})$ is a Priestley space, assume that $x\not\leq\alg{X} y$. By the Separation Theorem \cite[Theorem 1.4.4]{ClarkDavey1998}, there exists a $\var{X}_2$-morphism $\varphi \colon \alg{X} \to \utilde{\PL}_2$ with $\varphi(x) > \varphi(y)$. If $\varphi(x) = 1$, choose $U = \varphi^{-1}(\{ 1 \})$ and if $\varphi(x) = \frac{1}{2}$, choose $U = \varphi^{-1}{\{\frac{1}{2}\}} \cup \varphi^{-1}(\{ 1 \})$. In both cases, $U$ is a clopen (because $\PL_2$ carries the discrete topology and $\varphi$ is continuous) upset (because $\varphi$ is order-preserving) which contains $x$ but not $y$. 

To see (c), assume $x \ntriangleleft^\alg{X} y$ but $x\leq^\alg{X} y$. Then, again by the Separation Theorem, there exists a morphism $\varphi\colon \alg{X} \to \utilde{\PL}_2$ with $\varphi(x) \ntriangleleft \varphi(y)$ but $\varphi(x) \leq \varphi(y)$. Since $\lhd = {\leq}{\setminus}\{(\frac{1}{2},\frac{1}{2})\}$, this implies $\varphi(x) = \varphi(y) = \frac{1}{2}$. The clopen upset $U = \varphi^{-1}( \{ 1 \})$ and the clopen downset $D = \varphi^{-1}(\{ 0 \})$ satisfy the two subconditions of (c), the first one since $\varphi(x) = \varphi(y) = \frac{1}{2}$ and the second one since $z \lhd\alg{X} z'$ and $\varphi(z) =  \varphi(z') = \frac{1}{2}$ would yield a contradiction $\varphi(z) \ntriangleleft \varphi(z')$ to $\varphi$ being a morphism.   

For the converse, assuming that $\alg{X} = (X,\lhd^\alg{X}, \leq^\alg{X},\mathcal{T})$ satisfies (a)-(c), we want to show that it is a member of $\var{X}_2$. We apply the Separation Theorem again. 

Suppose $x \not\leq^\alg{X} y$. Using that $(X,\leq^\alg{X}, \mathcal{T})$ is a Priestley space, we can find a clopen upset $U$ which contains $x$ but not $y$. We define a continuous map $\varphi \colon X \to \{ 0, \frac{1}{2}, 1\}$ by $\varphi(z) = 1$ if $z \in U$ and $f(z) = 0$ otherwise. This clearly is order-preserving, and it also preserves $\lhd$, because $\lhd$ is a subset of $\leq$ by (a) and, in $\utilde{\PL}_2$ the relations $\lhd$ and $\leq$ coincide on the subset $\{ 0,1 \}$. Clearly this morphism satisfies $\varphi(x) \not\leq \varphi(y)$. 

In particular, the above covers the case where $x \neq y$ and the case where $x \ntriangleleft^\alg{X} y$ and $x \not\leq^\alg{X} y$ hold. Now assume $x \ntriangleleft^\alg{X} {y}$ but $x \leq^\alg{X} y$. Take a clopen upset $U$ and a clopen downset $D$ as given in (c). Replacing $U$ by the clopen upset $U' := U{\setminus}D$, the properties of (c) are still satisfied, since $z \lhd^\alg{X} z'$ and $z\notin D$ imply $z'\in U$, and $z'\in D$ would yield the contradiction $z \in D$, so $z' \in U'$. Let the continuous map $\varphi\colon X \to \{0 ,\frac{1}{2}, 1\}$ be defined via
\[
\varphi(z) = \begin{cases}
0 & \text{if } z \in D, \\
1 & \text{if } z\in U',\\
\frac{1}{2} & \text{if } z \in X{\setminus}(D \cup U').
\end{cases}
\]    
This is a well-defined continuous map since $D$, $U'$ and $X{\setminus}(D\cup C)$ forms a clopen partition of $X$. Furthermore, since $x\notin D$ (which implies $y\notin D$) and $y\notin U$ (which implies $x\notin U$) implies $\varphi(x) = \varphi(y) = \frac{1}{2}$ (\emph{i.e.}, $\varphi(x)\ntriangleleft \varphi(y)$), it remains to be shown that $\varphi$ preserves $\leq$ and $\lhd$. Order-preservation follows immediately from the fact that $U$ is an upset and $D$ is a downset. Now suppose $z \lhd^\alg{X} z'$. Then $z \in D$, which implies $\varphi(z) = 0$ holds, or $z' \in U'$, which implies $\varphi(z') = 1$ holds. In both cases, $\varphi(z) \lhd \varphi(z')$ is assured.              
\end{proof}

In the next subsection, we give a similar but more `implicit' axiomatization of the categories $\var{X}_n$ for $n > 2$ as well. Since (as we've already seen in the case $n=2$) all structures $\alg{X}\in \var{X}_n$ have underlying Priestley spaces, we then proceed to explore various functors relating our natural dualities to Priestley duality. 
\subsection{The relationship to Priestley duality}\label{subsec:RelationshipPriestley}
We continue to denote the functors  establishing the duality from Theorem~\ref{thm:NaturalDualityPositiveMVn} by $\func{D}_n \colon \PMV_n \to \var{X}_n$ and $\func{E}_n \colon \var{X}_n \to \PMV_n$. In particular, for $n = 1$ this coincides with Priestley duality between the variety of distributive lattices $\DL = \PMV_1$ and the category of Priestley spaces $\Priest = \var{X}_1$. In this case, we simply use $\func{D}\colon \DL \to \Priest$ and $\func{E}\colon \Priest \to \DL$ instead of $\func{D}_1$ and $\func{E}_1$. 

In this subsection, we show that there are functors $\Skel\colon \PMV_n \to \DL$ taking the \emph{distributive skeleton} and $\Pow\colon \DL \to \PMV_n$ taking a \emph{Priestley power} with $\Skel$ being left-adjoint to $\Pow$. This is similar to the adjunction between the Boolean skeleton functor $\MV_n \to \BA$ and the Boolean power functor $\BA \to \MV_n$ (which exists for any variety generated by a semi-primal lattice extension) from \cite[Section 4]{KurzPoigerTeheux2023}.   

While, in theory, the Separation Theorem \cite[Theorem 1.4.3]{ClarkDavey1998} always gives an `implicit' description of the dual categories, the reader can imagine that for $n>2$, it gets increasingly complicated to come up with more `explicit' descriptions of the categories $\var{X}_n$ similar to Theorem~\ref{thm:AxiomatizationPMV2}. Therefore, in these cases we content ourselves with the following.  

\begin{proposition}\label{prop:AximoatizationDualPMVn}
A structured Stone space $\alg{X} = \langle X, (\alg{R}^\alg{X} \mid \alg{R} \in \mathcal{S}_n), \mathcal{T}\rangle$ with closed binary relations $\alg{R}^\alg{X}$ is a member of $\var{X}_n$ if and only if it satisfies the following:
\begin{enumerate}[(a)]
\item $x \alg{R}_1^\alg{X} y \Rightarrow x \alg{R}_2^\alg{X} y$ for all $\alg{R}_1 \subseteq \alg{R}_2$ in $\mathcal{S}_n$. 
\item $\langle X, \leq^\alg{X} , \mathcal{T}\rangle$ is a Priestley space.
\item For all $\alg{R} \in \mathcal{S}_n{\setminus}\{ \leq \}$, if $(x,y) \notin \alg{R}^\alg{X}$, then there is a structure-preserving continuous map $\varphi\colon \alg{X} \to \utilde{\PL}_n$ with $(\varphi(x), \varphi(y)) \notin \alg{R}$. 
\end{enumerate}  
\end{proposition}
\begin{proof}
Every member $\alg{X}$ of $\var{X}_n$ satisfies the quasi-atomic formulas from (a). Furthermore, both (b) and (c) are immediate consequences of the Separation Theorem. To see that $\langle X, \leq^\alg{X}, \mathcal{T} \rangle$ is a Priestley space), assume $x \not\leq^{\alg{X}} y$. By the Separation Theorem there is a morphism $\varphi \colon \alg{X} \to \utilde{\PL}_n$ with $\varphi(x) \not\leq \varphi(y)$. Let $\varphi(x) = \frac{i}{n}$. Then $U = \varphi^{-1}(\{\frac{1}{n}\}) \cup \varphi^{-1}(\{\frac{i+1}{n}\}) \cup \dots \cup \varphi^{-1}(\{\frac{n-1}{n}\}) \cup \varphi^{-1}(\{1\})$ is a clopen upset which contains $x$ but not $y$. The converse is also a straightforward application of the Separation Theorem. 
\end{proof}

Therefore, there always is a forgetful functor $\U\colon \var{X}_n \to \Priest$ sending an object of $\var{X}_n$ to its underlying Priestley space and a $\var{X}_n$-morphism to itself. In the following, we show that the dual of $\U$ is given by the \emph{distributive skeleton functor} $\Skel\colon \PMV_n \to \DL$. This is similar to the \emph{Boolean skeleton functor} $\MV_n \to \BA$, which is dual to the corresponding forgetful functor from the category dual to $\MV_n$ to $\Stone$ \cite[Subsection 4.2]{KurzPoigerTeheux2023}.
The distributive skeleton of a $\PMV_n$ algebra is defined completely analogous to the Boolean skeleton of an $\MV_n$ algebra (see, \emph{e.g.}, \cite[Section 1.5]{Cignoli2000}. 
\begin{definition}\label{def:DistributiveSkeleton}
Let $\alg{A} \in \PMV_n$. The \emph{distributive skeleton} of $\alg{A}$ is the bounded distributive lattice 
\[
\Skel({\alg{A}}) = \langle \Skel(A), \wedge, \vee, 0, 1 \rangle 
\]
defined on the carrier set $\Skel(A) = \{ a\in A \mid a \oplus a = a\}$, with the operations $\wedge, \vee$ and constants $0,1$ inherited from $\alg{A}$.   
\end{definition}
To turn this into a functor $\Skel\colon \PMV_n\to\DL$, for a homomorphism $h \colon \alg{A} \to \alg{A}'$ between $\PMV_n$-algebras, simply define the homomorphism $\Skel h \colon \Skel(\alg{A}) \to \Skel(\alg{A}')$ by restriction $\Skel h = h{\mid}_{\Skel(\alg{A})}$.    

\begin{theorem}\label{thm:SkeletonDualForgetful}
The functor $\Skel\colon \PMV_n \to \DL$ is dual to the functor $\U\colon \var{X}_n \to \Priest$, that is, $\func{D}\Skel$ is naturally isomorphic to $\U\func{D}_n$.  
\end{theorem}

\begin{proof}
By definition, natural in the choice of $\alg{A} \in \PMV_n$, we want to find an order-preserving homeomorphism 
\[
\Phi_\alg{A}\colon (\PMV_n(\alg{A}, \PL_n), \leq) \to (\DL(\Skel(\alg{A}), \alg{2}), \leq),  
\]
where $\alg{2}$ denotes the two-element distributive lattice. We claim that 
\[ 
\Phi_\alg{A}(u) = u{\mid}_{\Skel(\alg{A})}
\]
has these desired properties.  

To see that $\Phi_\alg{A}$ is injective, suppose that $u \neq u'$ are two distinct homomorphisms $\alg{A} \to \PL_n$. Let $a\in \alg{A}$ be such that $u(a) \neq u'(a)$, without loss of generality say $u(a) < u'(a)$. Then, for $d = u'(a)$, we have $u(\tau_d(a)) = \tau_d(u(a)) = 0$ and $u'(\tau_d(a)) = \tau_d(u'(a)) = 1$. Since $\tau_d(a) \in \Skel(\alg{A})$ holds, this shows that $\Phi(u) \neq \Phi(u')$. 

Now we show that $\Phi$ is surjective. Let $p\colon \Skel(\alg{A}) \to \alg{2}$ be a homomorphism. We construct a homomorphism $u_p \colon \alg{A} \to \PL_n$ with $\Phi_\alg{A}(u_p) = p$. Given $a\in \alg{A}$, define 
\[
u_p(a) = \bigvee \{ d \mid p(\tau_d(a)) = 1\}. 
\] 
Clearly $u_p$ preserves $0$ and $1$. Now let $a_1, a_2 \in \alg{A}$, let $u_p(a_1) = d_1$ and $u_p(a_2) = d_2$. We want to show that, for $\ast \in \{ \wedge, \vee, \odot, \oplus \}$, $u_p(a_1 \ast a_2) = d_1 \ast d_2$. In other words, we want to show that
$p(\tau_{d_1 \ast d_2}(a_1 \ast a_2)) = 1$ and $p(\tau_{d'}(a_1 \ast a_2)) = 0$ for all $d' > d_1 \ast d_2$. Since $\ast$ is order-preserving we know that $\PL_n$ satisfies 
\[
\tau_{d_1}(x_1) \wedge \tau_{d_2}(x_2) \leq \tau_{d_1 \ast d_2}(x_1 \ast x_2).
\]   
Since this can be expressed as an equation, it also holds in $\alg{A}$. Therefore, we get 
\[
1 = p(\tau_{d_1}(a_1) \ast \tau_{d_2}(a_2)) \leq p(\tau_{d_1 \ast d_2}(a_1 \ast a_2)).
\]
Now let $d' > d_1 \ast d_2$. Then, since $d_1 \ast d_2 \neq 1$, we can choose minimal $d_1' > d_1$ and $d_2' \geq d_2$ with $d_1' \ast d_2' \geq d'$. By minimality, $\PL_n$ satisfies the equation corresponding to  
\[
\tau_{d_1}(x_1) \wedge \tau_{d_2}(x_2) \wedge \tau_{d'}(x_1\ast x_2) \leq \tau_{d_1'}(x_1),
\]
which is therefore also satisfied in $\alg{A}$. But now, if we assume that $p(\tau_{d'}(a_1\ast a_2)) = 1$, then 
\[
1 = p(\tau_{d_1}(a_1) \wedge \tau_{d_2}(a_2) \wedge \tau_{d'}(a_1\ast a_2)) \leq p(\tau_{d_1'}(a_1))
\]
implies $p(\tau_{d'_1}(a_1)) = 1$, which is a contradiction to $u_p(a_1) = d_1$. Therefore, $u_p$ is a homomorphism. The restriction of $u_p$ to $\Skel(\alg{A})$ is equal to $p$ because $a\in \Skel(\alg{A})$ is equivalent to $\tau_d(a) = a$ for all $d\in \PL_n{\setminus}\{ 0 \}$. 

Thus we showed that $\Phi_\alg{A}$ is bijective. It is also continuous, and therefore a homeomorphism, since a subbasis of the topology on $\func{D}_1\Skel(\alg{A})$ is given by the sets of the form $[a:e] = \{ p \colon \Skel(\alg{A}) \to \alg{2} \mid p(a) = e\}$ where $a$ ranges over $\Skel(\alg{A})$ and $e$ ranges over $\alg{2}$. The preimage $\Phi^{-1}([a:e])$ is exactly the corresponding subbase element $[a:e] = \{ h \colon \alg{A} \to \PL_n \mid h(a) = e \}$ of the topology on $\U \func{D}_n(\alg{A})$. The fact that $\Phi_\alg{A}$ is order-preserving follows directly from its definition, so it only remains to show that $\Phi$ defines a natural transformation $\U \func{D}_n \Rightarrow \func{D}_1 \Skel$. Let $h\colon \alg{A} \to \alg{A}'$ be a homomorphism. We need to show that the square 
\[
\begin{tikzcd}[row sep=4.5em,column sep=4.5em]
\PMV_n(\alg{A'}, \PL_n) \arrow[r, "\Phi_{\alg{A}'}"] \arrow[d, "\func{UD}_nh"'] & \DL(\Skel(\alg{A'}),\alg{2}) \arrow[d, "\func{D}_1\Skel h"] \\
\PMV_n(\alg{A}, \PL_n) \arrow[r, "\Phi_{\alg{A}}"'] &  \DL(\Skel(\alg{A'}),\alg{2}) 
\end{tikzcd}
\] 
commutes. By definition, for a homomorphism $u \colon \alg{A}' \to \PL_n$ we have 
\[
\Phi_\alg{A} \circ \U\func{D}_n h (u) = \Phi_\alg{A} (u \circ h) = (u \circ h){\mid}_{\Skel(\alg{A})} 
\]     
and 
\[
\func{D}_1\Skel h \circ \Phi_{\alg{A}'}(u) = \func{D_1}\Skel h (u{\mid}_\Skel{\alg{A'}}) = u{\mid}_{\Skel(\alg{A'})} \circ h{\mid}_{\Skel(\alg{A})},  
\]
which makes it easy to see that these two coincide, finishing the proof.  
\end{proof}

The Boolean skeleton functor $\MV_n \to \BA$ has a right-adjoint \cite[Subsection 4.3]{KurzPoigerTeheux2023}, which takes a Boolean algebra $\alg{B}$ to the \emph{Boolean power} $\lucas_n [\alg{B}]$ (see, \emph{e.g.}, \cite{Burris1975, BurrisSankappanavar1981} for information about Boolean powers). In the following we show that, similarly, the distributive skeleton functor has a right-adjoint, which takes the \emph{Priestley power} defined as follows.  
\begin{definition}\label{def:PriestleyPower}
Let $\alg{L} \in \DL$ be a distributive lattice and let $\alg{M}$ be a finite ordered algebra. The \emph{Priestley power}, $\alg{M}[\alg{L}]$, is given by the collection 
\[
\alg{M}[\alg{L}] = \Priest(\func{D}(\alg{L}), (M, \leq, \mathcal{T}_\mathrm{dis}))
\] of continuous order-preserving maps from the dual of $\alg{L}$ to the discrete Priestley space $(M, \leq, \mathcal{T}_\mathrm{dis})$.     
\end{definition} 
A more constructive defintion of Priestley powers is given and shown to be equivalent to the above definition in \cite{Lenkehegyi1986} (where they are called \emph{distributive extensions}). We also emphasize that our notion of Priestley power differs from the one established in \cite{Jipsen2009}.  

Similarly to the Boolean power (but with the constraint that all operations of $\alg{M}$ need to be order-preserving), we get the following. 

\begin{lemma}\label{lem:PriestleyPowerIsAlgebra}
Let $\alg{M}$ be a finite ordered algebra, all of whose operations are order-preserving. Then, for every distributive lattice $\alg{L} \in \DL$, the Priestley power $\alg{M}[\alg{L}]$ with component-wise operations is a subalgebra of $\alg{M}^{\func{D}(\alg{L})}$. 
\end{lemma} 

\begin{proof}
Let $f$ be an $n$-ary operation of $\alg{M}$ and let $\alpha_1, \dots, \alpha_n \in \alg{M}[\alg{L}]$. We need to show that $\alpha \colon \func{D}(\alg{L}) \to M$ defined by 
$\alpha(x) = f(\alpha_1(x),\dots, \alpha_n(x))$ is continuous and order-preserving. Order-preservation is easy, since if $x \leq y$ we know that $\alpha_i(x) \leq \alpha_i(y)$ for all $i$ and since $f$ is order-preserving we have 
\[
\alpha(x) = f(\alpha_1(x), \dots, \alpha_n(x)) \leq f(\alpha_1(y), \dots, \alpha_n(y)) \leq \alpha(y).
\]
To see that $\alpha$ is continuous, we show that $\alpha^{-1}(\{ m \})$ is clopen for every $m \in M$. Let $N \subseteq M^n$ be the finite set of tuples $(m_1,\dots, m_n)$ with $f(m_1,\dots, m_n) = m$. Then we have 
\[
\alpha^{-1}(\{ m \}) = \bigcup_{(m_1,\dots,m_n)\in N} \alpha_1^{-1}(\{ m_1\}) \cap \dots \cap \alpha_n^{-1}(\{ m_n \}),
\] 
which is clopen because $N$ is finite and all $\alpha_i$ are continuous. 
\end{proof} 

Therefore, it is easily seen that the following \emph{Priestly power functor} $\Pow\colon \DL \to \PMV_n$ is well-defined. For a distributive lattice $\alg{L} \in \DL$ let $\Pow(\alg{L}) = \PL_n[\alg{L}]$ be the Priestley power, and for a homomorphism $h\colon \alg{L}_1 \to \alg{L}_2$ let $\Pow h\colon \Pow(\alg{L}_1) \to \Pow(\alg{L}_2)$ be defined by $\alpha \mapsto \alpha \circ \func{D} h$. We now show by duality that this functor is right-adjoint to the distributive skeleton functor. 

\begin{theorem}\label{thm:PowerSkeletonAdjoint}
The Priestley power functor $\Pow \colon \DL \to \PMV_n$ is right-adjoint to the distributive skeleton functor $\Skel\colon \PMV_n \to \DL$.  
\end{theorem} 

\begin{proof}
Our proof strategy consists of the following two steps. We first define a functor $\func{P}\colon \Priest \to \var{X}_n$ and show that it is left-adjoint to the forgetful functor $\U_n$. Then we show that $\Pow$ is the dual of $\func{P}$. By Theorem~\ref{thm:SkeletonDualForgetful} and the uniqueness of an adjoint up to natural isomorphism, the theorem follows. 
\[
\xymatrix@C=100pt@R=50pt{ 
\var{X}_n \ar@/^/[r]^{\func{E}_n} 
\ar@{<-}@/_8pt/[d]_{\func{P}}^{\phantom{.}\dashv}
\ar@/^8pt/[d]^{\func{U}}
& 
\PMV_n \ar@/^/[l]^{\func{D}_n} 
\ar@{<-}@/_8pt/[d]_{\Pow}
\ar@/^8pt/[d]^{\Skel}_{\vdash\phantom{.}}
\\
\Priest \ar@/^/[r]^{\func{E}}
&
\DL \ar@/^/[l]^{\func{D}} 
}
\]
Let $\func{P}\colon \Priest \to \var{X}_n$ be defined as follows. For a Priestley space $\langle X, \leq \rangle$ , define $\func{P}(X,\leq)$ to be the structured topological space $\langle X, \leq, (\alg{R}^\alg{X} = \emptyset \mid \alg{R}\in \mathcal{S}_n{\setminus}\{ \leq \} ) \rangle$, which is a well-defined member of $\var{X}_n$ by Proposition~\ref{prop:AximoatizationDualPMVn}. Furthermore, define $\func{P}\varphi = \varphi$ on morphisms. It is easy to see that $\func{P}$ is left-adjoint to $\U$, since for every Priestley space $\langle X, \leq \rangle$ and structure $\alg{Y} \in \var{X}_n$ because, by definition of $\func{P}$, morphisms in $\var{X}_n(\func{P}(X,\leq), \alg{Y})$ 
clearly coincide with continuous order-preserving maps $X \to Y$, that is, morphisms in $\Priest(\langle X, \leq \rangle, \U(\func{Y}))$. 

We now show that $\func{P}$ is dual to $\Pow$, more specifically, we show that there is a natural isomorphism $\func{E}_n\func{P} \cong \Pow\func{E}$. For this, we simply note that, for a Priestley space $\langle X,\leq \rangle$, we have the following natural isomorphisms
\begin{align*}
\func{E}_n\func{P}(X,\leq) & = \var{X}_n(\func{P}(X,\leq), \utilde{\alg{PL}}_n) \cong \\ &\cong \Priest(\langle X,\leq \rangle, \func{U}(\utilde{\PL}_n)) \cong \\
& \cong \Priest(\func{DE}(X,\leq), \func{U}(\utilde{\PL}_n))
= \Pow\func{E}(X,\leq),
\end{align*}
where we used $\func{P} \dashv \U$ established above and the definition of the Priestley power $\Pow(\alg{L}) = \Priest(\func{D}(\alg{L}), \U(\utilde{\PL}_n))$. This finishes the proof.
\end{proof}

One simple consequence of (the proof of) Theorem\ref{thm:PowerSkeletonAdjoint} is the following. 

\begin{corollary}\label{cor:EmbedPowerSkeleton}
Every algebra $\alg{A} \in \PMV_n$ is a subalgebra of a Priestley power. More specifically, there is an embedding $\alg{A} \hookrightarrow \Pow\Skel(\alg{A})$. 
\end{corollary}

\begin{proof}
Let $\func{P}$ be the dual of $\Pow$ as in the proof of Theorem~\ref{thm:PowerSkeletonAdjoint}. It is easy to see that the counit of the adjunction $\func{P} \dashv \func{U}$ is the identity map $id_x$ as a morphism $\func{P}\func{U}(\alg{X}) \to \alg{X}$ on every component. Therefore, it is a component-wise epimorphism in $\var{X}_n$. Dually, this implies that the unit of the adjunction $\Skel \dashv \Pow$ is a component-wise monomorphism, and therefore yields an embedding $\alg{A} \hookrightarrow \Pow\Skel(\alg{A})$ for every $\PMV_n$-algebra $\alg{A} \in \PMV_n$ as desired. 
\end{proof}

In the last subsection, we describe the algebraically and existentially closed members of $\PMV_n$ via their duals. For this, Boolean powers (rather than Priestley powers) play an essential role. However, since Boolean powers arise as special cases of Priestley powers, the results of this subsection will prove useful towards this end. 

\subsection{Algebraically and existentially closed algebras}\label{subsec:A-E-ClosedAlgebras}
A standard application of natural dualities is the classification of algebraically closed and existentially closed algebras via their duals (see, \emph{e.g.}, \cite[Sections 5.3 and 5.4]{ClarkDavey1998}). In this subsection, we give full classifications of the algebraically closed and existentially closed members of $\PMV_n$ via Boolean powers. Note that, for a complemented bounded distributive lattice $\alg{B}$, the Priestley power $\PL_n[\alg{B}]$ from Definition~\ref{def:PriestleyPower} coincides with the usual Boolean power $\PL_n[\alg{B}]$. Since the structure $\utilde{\PL}_n$ is total, we can use the AC-EC Theorem \cite[Theorem 5.3.5.]{ClarkDavey1998} to characterize algebraically and existentially closed members of $\PMV_n$. 

Before we state this theorem, we recall that $\alg{X} \in \var{X}_n$ has the \emph{dual finite homomorphism property} $(\text{FHP})^\ast$ if, for all finite $\alg{Y}, \alg{Z} \in \var{X}_n$ and surjective morphisms $\varphi\colon \alg{X} \to \alg{Z}$, $\psi\colon \alg{Y} \to \alg{Z}$, there exists a morphism $\lambda\colon \alg{X} \to \alg{Y}$ such that $\varphi = \psi \circ \lambda$. 
\begin{equation}
\tag*{$\text{(FHP})^\ast$}
\begin{tikzcd}
\alg{X} \arrow[r, twoheadrightarrow, "\varphi"] \arrow[dr, dashrightarrow, "\exists \lambda"']
& \alg{Z} \\
& \alg{Y} \arrow[u, twoheadrightarrow, "\psi"']
\end{tikzcd}
\end{equation}
The \emph{dual finite embedding property} $(\text{FEP})^\ast$ is the similar same, except that $\lambda$  is also required to be surjective.      
\begin{theorem}\label{thm:AC-EC-Thm}\cite{ClarkDavey1998}
Let $\alg{A} \in \PMV_n$. 
\begin{enumerate}
\item $\alg{A}$ is algebraically closed if and only if $\func{D}_n(\alg{A})$ has the dual finite homomorphism property $(\text{FHP})^\ast$. 
\item $\alg{A}$ is existentially closed if and only if $\func{D}_n(\alg{A})$ has the finite embedding property $(\text{FEP})^\ast$.
\end{enumerate}
\end{theorem}

We now show that algebraically and existentially closed members of $\PMV_n$ stem from Boolean algebras in the following sense.

\begin{theorem}\label{thm:AlgebraicallyExistentiallyClosed}
Let $\alg{A} \in \PMV_n$. 
\begin{enumerate}
\item $\alg{A}$ is algebraically closed if and only if $\alg{A}$ is isomorphic to a Boolean power $\PL_n[\alg{B}]$, where $\alg{B} \in \BA$ is an arbitrary Boolean algebra.  
\item $\alg{A}$ is existentially closed if and only if $\alg{A}$ is isomorphic to a Boolean power $\PL_n[\alg{B}]$, where $\alg{B} \in \BA$ is an atomless Boolean algebra.
\end{enumerate}
\end{theorem}

\begin{proof}
By (the proof of) Theorem~\ref{thm:PowerSkeletonAdjoint}, we know that the duals of Boolean powers $\PL_n[\alg{B}]$ in $\var{X}_n$ are exactly the structures isomorphic to some $\alg{X} \in \var{X}_n$ where $\leq^{\alg{X}}$ the discrete order and $\alg{R}^\alg{X}$ is empty for all other $\alg{R}\in \mathcal{S}_n$. We first show by contrapositive that if $\alg{X}$ has the finite homomorphism property, then it needs to be of this form.  

Let $\alg{X} \in \var{X}_n$ not be of the form described above. If $\leq$ is not discrete, there are distinct $x, y\in \alg{X}$ with $x < y$. Let $U$ be an upset of $X$ containing $y$ but not $x$. Define $\alg{Z} \in \var{X}_n$ to consist of two points $\{ a,b \}$ with order $a < b$ (and all other relations empty), and let $\alg{Y}$ consist of two points $\{ a', b' \}$ with the discrete order (and all other relations empty). Let $\varphi\colon \alg{X} \to \alg{Z}$ be the morphism sending $U$ to $b$ and $X{\setminus}U$ to $a$. Let $\psi\colon \alg{Y} \to \alg{Z}$ be the morphism sending $a'$ to $a$ and $b'$ to $b$. Now if there was a morphism $\lambda$ witnessing $(\text{FHP})^\ast$, it would have to satisfy $\lambda(x) = a'$ and $\lambda(y) = b'$. However, this is impossible since this would mean $x\leq y$ and $\lambda(x) \not\leq \lambda(y)$, contradicting that $\lambda$ needs to be order-preserving. 

Now assume that $\alg{X}$ has the discrete order-relation and there is some other relation $\alg{R}^\alg{X}$ which is non-empty. Choose $\alg{R}$ minimal in $\mathcal{S}_n$ such that there is some $x\in \alg{X}$ with $x\alg{R}^\alg{X}x$. Define $\alg{Z} \in \var{X}_n$ to consist of one point $\{ a \}$ with $a\alg{R}^\alg{X}a$ and let $\alg{Y}$ consists of one point $\{ a' \}$ with $a' \leq^\alg{Y} a'$ and all other $\alg{R}^\alg{Y}$ empty. Let $\varphi \colon \alg{X} \to \alg{Z}$ and $\psi\colon \alg{Y} \to \alg{Z}$ be the unique morphisms. The unique map $\lambda \colon X \to Z$ is not a morphism because otherwise $x \alg{R}^\alg{X} x$ would imply $a'\alg{R} a'$. Therefore $\alg{X}$ does not satisfy $(\text{FHP})^\ast$. 

Thus we showed that if $\alg{A}$ is algebraically closed, then it is isomorphic to some Boolean power $\PL_n[\alg{B}]$. For the converse of (1), one has to show that every $\alg{X} \in \var{X}_n$ with discrete order and all other relations empty has the finite homomorphism property. For (2), one has to show such an $\alg{X}$ has the finite embedding property if and only if $\leq$ has no isolated points. However, this is easy, since both of these can be proven completely analogous to \cite[Theorem 5.4.1]{ClarkDavey1998}.
\end{proof}

In particular, for $n = 1$ we recover the well-known description of algebraically closed and existentially closed distributive lattices \cite{Schmid1979} as complemented distributive lattices and atomless complemented distributive lattices.   

\section{Conclusion}\label{sec:conclusion}
We developed a logarithmic optimal natural duality for the variety $\PMV_n$ of positive $\MV_n$-algebras, generated by $\PL_n$, the negation-free reduct of the finite $\MV$-chain $\lucas_n$. We explored the relationship between this duality and Priestley duality, showing that there is an adjunction between $\DL$ and $\PMV_n$ given by the distributive skeleton functor $\Skel\colon \PMV_n \to \DL$ and the Priestley power functor $\Pow\colon \DL \to \PMV_n$. Specializing this relationship to Boolean powers, we gave a full characterization of algebraically and existentially closed members of $\PMV_n$. In the following, a few open questions, remarks and ideas for further research are collected. 
\begin{enumerate}
\item As noted in Remark~\ref{rem:OtherAlgebrasWithTaus}, the results of Subsection~\ref{subsec:NaturalDualitiesPosMV} up until Lemma~\ref{lem:minimalrelation}, as well as  Lemma~\ref{lem:subalgebrasSimple} and Corollary~\ref{cor:HSP=ISP} from Subsection~\ref{subsec:PositiveMVChains} hold not only for $\PL_n$, but for every lattice-based algebra $\alg{D}$ in which $\tau_d$ is term-definable for every $d\in\alg{D}$. Furthermore, the existence of these $\tau_d$ together with a `weak form of negation' is equivalent to semi-primality of a lattice-based algebra \cite[Proposition 2.8]{KurzPoigerTeheux2023}. Is there a sensible definition of `lattice-semi-primal' algebras (similar to lattice-primal algebras), such that the results of this paper be seen as a specific instance of a more general result about natural dualities for such lattice-semi-primal algebras? 
\item In Corollary~\ref{cor:EmbedPowerSkeleton}, we showed that every $\PMV_n$-algebra can be embedded into the Priestley power of its distributive skeleton, similarly to how every $\MV_n$-algebra can be embedded into the Boolean power of its Boolean skeleton. Based on this, a category equivalent to $\MV_n$ was described in \cite{DiNolaLettieri2000}. This equivalence was explained from the point of view of natural dualities in \cite{Niederkorn2001}. Is there a similar categorical equivalence for $\PMV_n$? To deal with this question, the more constructive description of Priestley powers from \cite{Lenkehegyi1986} could prove useful. 
\item In this paper, the logical aspects of $\PMV_n$-algebras were only hinted at. In future work, we plan to explore these aspects further. In particular, we aim to investigate modal extensions of $\PMV_n$-algebras to deal with an analogue of Dunn's positive modal logic \cite{Dunn1995} in the setting of modal finitely-valued \L ukasiewicz logic \cite{HansoulTeheux2013}.     
\end{enumerate}      
\appendix
\section{An algorithm to find the dualizing structure}\label{appendix}
Making use of Proposition~\ref{prop:subalgebrasAsUnionsOfRectangles}, here we provide an algorithm to find the lattice $\mathcal{S}_n$ from Definition~\ref{def:SetOfRelevantRelations}, in order to systematically determine the alter ego from Theorem~\ref{thm:NaturalDualityPositiveMVn}. 

By Proposition~\ref{prop:subalgebrasAsUnionsOfRectangles}, the subalgebras $\lhd \subseteq \alg{R} \subseteq \leq$ are certain `unions of rectangles' of the form 
\[
\bigcup_{i = 0}^n C_{(\frac{i}{n},y_i)}
\]
where $y_0 \leq \dots \leq y_n$ is an increasing sequence in $\PL_n$ with $\frac{i}{n} \leq y_i$. We will identify this union with the sequence $y_1, \dots, y_{n-1}$ (note that $y_0 = 0$ and $y_n = 1$ holds for every subalgebra, so we can omit these). For example, the sequence where all $y_i = 1$ corresponds to $\lhd$ and the sequence $y_i = \frac{i}{n}$ corresponds to $\leq$. However, not every such sequence corresponds to a subalgebra. By Proposition~\ref{prop:subalgebrasAsUnionsOfRectangles}, the sequences which correspond to subalgebras, which we call \emph{good} sequences, are the ones which satisfy
\[
(\tfrac{j}{n}, y_j) \odot (\tfrac{j'}{n}, y_{j'}) \in \bigcup_{i = 0}^n C_{(\frac{i}{n},y_i)} \text{ and } (\tfrac{j}{n}, y_j) \oplus (\tfrac{j'}{n}, y_{j'}) \in \bigcup_{i = 0}^n C_{(\frac{i}{n},y_i)}
\]
for all $j,j' \in \{ 1,\dots, n-1\}$. In fact, similar to the proof of Lemma~\ref{lem:SubdirectProductsSuffice}, it can be shown that it is sufficient to check this condition only for the `corner elements', that is, at indices $i$ which satisfy $y_i < y_{i+1}$ (with $y_n = 1$).           

Thus, the set $\mathcal{S}_n$ is in bijective correspondence with the set of good sequences. Furthermore, the union associated to $y_1,\dots, y_{n-1}$ is contained in the union associated to $y_1',\dots,y_{n-1}'$ if and only if $y_i' \leq y_i$ holds for $i = 1,\dots, n-1$. Thus we can also retrieve the lattice structure of $\mathcal{S}_n$. Altogether, we proved that the following algorithm yields the lattice $\mathcal{S}_n$. 
 
\textbf{Step 1.} Generate the set $Y$ of all sequences $[y_1,\dots,y_n]$ of elements of $\PL_n$ with $y_1 \leq, \dots, \leq y_{n-1}$ and $\frac{i}{n} \leq y_i$ for all $i = 1,\dots, n-1$. 

\textbf{Step 2.} Start with $G = \emptyset$ and do the following for every sequence $[y_1,\dots, y_{n-1}] \in Y$. Let $J \subseteq \{ 1,\dots, n-1\}$ be the collection of all indices $i$ with $y_i < y_{i+1}$ (where $y_n := 1$). Check whether  
\[
(\tfrac{j}{n}, y_j) \odot (\tfrac{j'}{n}, y_{j'}) \in \bigcup_{i = 0}^n C_{(\frac{i}{n},y_i)} \text{ and } (\tfrac{j}{n}, y_j) \oplus (\tfrac{j'}{n}, y_{j'}) \in \bigcup_{i = 0}^n C_{(\frac{i}{n},y_i)}
\]
holds for all $j,j' \in J$. Add $[y_1,\dots, y_{n-1}]$ to $G$ if and only if this holds.

\textbf{Step 3.} Order the set $G$ obtained after completing Step 2 by 
\[
[y_1,\dots, y_{n-1}] \leq [y_1',\dots,y'_{n-1}] \Leftrightarrow y'_i \leq y_i \text{ for } i = 1,\dots, n-1.
\]
This results in $G \cong \mathcal{S}_n$ (note that the order on $G$ is simply the component-wise converse order $\geq$).

As an example, we compute the set $\mathcal{S}_4$ using the above algorithm. 

\textbf{Step 1.} The set $Y$ consists of the following $14$ sequences: 
\begin{align*}
[1,1,1], &&  
[\tfrac{3}{4},1,1], && 
[\tfrac{3}{4},\tfrac{3}{4},1], && 
[\tfrac{3}{4}, \tfrac{3}{4}, \tfrac{3}{4}], && 
[\tfrac{2}{4},1,1], && 
[\tfrac{2}{4},\tfrac{3}{4},1], &&
[\tfrac{2}{4}, \tfrac{3}{4}, \tfrac{3}{4}], \\
[\tfrac{2}{4},\tfrac{2}{4},1], &&  
[\tfrac{2}{4}, \tfrac{2}{4}, \tfrac{3}{4}], && 
[\tfrac{1}{4},1,1], && 
[\tfrac{1}{4},\tfrac{3}{4},1], && 
[\tfrac{1}{4}, \tfrac{3}{4}, \tfrac{3}{4}], && 
[\tfrac{1}{4}, \tfrac{2}{4}, 1], &&
[\tfrac{1}{4}, \tfrac{2}{4}, \tfrac{3}{4}].
\end{align*}
 
\textbf{Step 2.} 
The set $G$ of good sequences consists of the following $7$ sequences:
\begin{align*}
[1,1,1], &&  
[\tfrac{3}{4},1,1], && 
[\tfrac{3}{4},\tfrac{3}{4},1], && 
[\tfrac{2}{4},1,1], && 
[\tfrac{2}{4},\tfrac{3}{4},1], && 
[\tfrac{2}{4},\tfrac{2}{4},1], &&
[\tfrac{1}{4}, \tfrac{2}{4}, \tfrac{3}{4}].
\end{align*}
The other sequences are not good for the following reasons.
\begin{itemize}
\item The sequences $[\tfrac{3}{4}, \tfrac{3}{4}, \tfrac{3}{4}]$ and $[\tfrac{2}{4}, \tfrac{3}{4}, \tfrac{3}{4}]$ are not good because $(\frac{3}{4}, \frac{3}{4}) \odot (\frac{3}{4}, \frac{3}{4}) = (\frac{2}{4}, \frac{2}{4}) \notin C_{(\frac{2}{4},\frac{2}{4})}$.
\item The sequence $[\frac{2}{4},\frac{2}{4},\frac{3}{4}]$ is not good because $(\frac{2}{4}, \frac{2}{4}) \odot (\frac{3}{4}, \frac{3}{4}) = (\frac{1}{4}, \frac{1}{4}) \notin C_{(\frac{1}{4},\frac{2}{4})}$.
\item The sequences $[\tfrac{1}{4},1,1]$, $[\tfrac{1}{4},\tfrac{3}{4},1]$ and $[\tfrac{1}{4}, \tfrac{3}{4}, \tfrac{3}{4}]$ are not good because $(\frac{1}{4}, \frac{1}{4}) \oplus (\frac{1}{4}, \frac{1}{4}) = (\frac{2}{4}, \frac{2}{4}) \notin C_{(\frac{2}{4},\frac{3}{4})}$. 
\item The sequence $[\tfrac{2}{4},\tfrac{2}{4},1]$ is not good because $(\frac{1}{4}, \frac{1}{4}) \oplus (\frac{2}{4}, \frac{2}{4}) = (\frac{3}{4}, \frac{3}{4}) \notin C_{(\frac{3}{4},1}$.
\end{itemize} 

\textbf{Step 3.} The lattice $\mathcal{S}_4$ looks as follows: 
\[
\begin{tikzpicture}
  \node (0) at (0,0) {$[1,1,1]$};
  \node (1) at (0,1) {$[\tfrac{3}{4},1,1]$};
  \node (2) at (-1,2){$[\tfrac{2}{4},1,1]$};
  \node (3) at (1,2) {$[\tfrac{3}{4},\tfrac{3}{4},1]$};
  \node (4) at (0,3) {$[\tfrac{2}{4},\tfrac{3}{4},1]$};
  \node (5) at (0,4) {$[\tfrac{2}{4},\tfrac{2}{4},1]$};
  \node (6) at (0,5) {$[\tfrac{1}{4}, \tfrac{2}{4}, \tfrac{3}{4}]$};
  \draw (0) -- (1) -- (2) -- (4) -- (5) -- (6);
  \draw (1) -- (3) -- (4); 
\end{tikzpicture}
\]
To obtain an optimal strong duality for $\PMV_4$, we only need to consider the meet-irreductible elements of this lattice, \emph{i.e.}, we remove the relation corresponding to the sequence $[\tfrac{3}{4},1,1]$.   

\section*{Acknowledgment}
  \noindent The author is supported by the Luxembourg National Research Fund under the project  PRIDE17/12246620/GPS.

\end{document}